\documentclass[12pt,reqno]{amsart}
\usepackage{fullpage}
\usepackage{times}
\usepackage{colonequals}
\usepackage{amsmath,amssymb,amsthm,url}
\usepackage{xcolor}
\usepackage[alphabetic]{amsrefs}
\usepackage[utf8]{inputenc}
\usepackage[english]{babel}
\usepackage{comment}
\usepackage{bbm}
\usepackage{enumerate}
\usepackage{bm}
\usepackage{graphicx}
\usepackage{mathrsfs}
\usepackage[colorlinks=true, pdfstartview=FitH, linkcolor=blue, citecolor=blue, urlcolor=blue]{hyperref}
\allowdisplaybreaks
\usepackage{caption} 
\usepackage[vlined,ruled]{algorithm2e}

\newtheorem{thm}{Theorem}[section]
\newtheorem{problem}[thm]{Problem}
\newtheorem{lem}[thm]{Lemma}
\newtheorem{prop}[thm]{Proposition}
\newtheorem{defn}[thm]{Definition}
\newtheorem{ex}[thm]{Example}
\newtheorem{cor}[thm]{Corollary}
\newtheorem{rem}[thm]{Remark}

\newtheorem{conj}[thm]{Conjecture}

\newcommand{\Tr}{\operatorname{Tr}}
 
\newcommand{\N}{\mathbb{N}}
\newcommand{\R}{\mathbb{R}} 
\newcommand{\Z}{\mathbb{Z}}
\newcommand{\F}{\mathbb{F}}

\AtBeginDocument{%
	\def\MR#1{}
}

\begin{document}

\title{The subspace structure of maximum cliques in pseudo-Paley graphs from unions of cyclotomic classes}
\author{Shamil Asgarli}
\address{Department of Mathematics and Computer Science \\ Santa Clara University \\ 500 El Camino Real \\ USA 95053}
\email{sasgarli@scu.edu}
\author{Chi Hoi Yip}
\address{Department of Mathematics \\ University of British Columbia \\ 1984 Mathematics Road \\ Canada V6T 1Z2}
\email{kyleyip@math.ubc.ca}

\begin{abstract}
Blokhuis showed that all maximum cliques in Paley graphs of square order have a subfield structure. Recently, it has been shown that in Peisert-type graphs, all maximum cliques are affine subspaces, and yet some maximum cliques do not arise from a subfield. In this paper, we investigate the existence of a clique of size $\sqrt{q}$ with a subspace structure in pseudo-Paley graphs of order $q$ from unions of semi-primitive cyclotomic classes. We show that such a clique must have an equal contribution from each cyclotomic class and that most such pseudo-Paley graphs do not admit such cliques, suggesting that the Delsarte bound $\sqrt{q}$ on the clique number can be improved in general. We also prove that generalized Peisert graphs are not isomorphic to Paley graphs or Peisert graphs, confirming a conjecture of Mullin.
\end{abstract}

\maketitle

\section{Introduction}
Throughout the paper, $p$ denotes an odd prime, $q$ denotes a positive power of $p$, and $\F_q$ denotes the finite field with $q$ elements. For a finite field $\F_q$, we write $\F_q^+$ for its additive group and $\F_q^*$ for its multiplicative group. We always assume that $g$ is a fixed primitive root of $\F_q$. 

Let $q \equiv 1 \pmod 4$. The \emph{Paley graph} of order $q$, denoted $P_q$, is the Cayley graph $\operatorname{Cay}(\F_q^+,(\F_q^*)^2)$, where $(\F_q^*)^2$ is the set of squares in $\F_q^*$. Paley graphs are well-studied, connecting many branches of mathematics \cite{J20}, notably combinatorics and number theory.

It is known that Paley graphs are strongly regular, and the set of squares forms a particular cyclotomic class. In fact, Paley graphs belong to a larger family of semi-primitive cyclotomic strongly regular graphs, first constructed by Brouwer, Wilson, and Xiang~\cite{BWX}. In this paper, we study the intermediate family, namely, pseudo-Paley graphs (that is, graphs that share the same spectrum with some Paley graphs) arising from unions of semi-primitive cyclotomic classes. We study their similarity to Paley graphs regarding the structure of maximum cliques. The precise results will be in Section~\ref{subsect:mainresults}, and so we continue with the necessary definitions and further background. 

Let $N \mid (q-1)$. Let $C_0$ be the subgroup of $\F_q^*$ with index $N$, and let $C_1, \ldots, C_{N-1}$ be all the cosets of $C_0$, where $C_j=g^j C_0$. The sets $C_0, C_1, \ldots, C_{N-1}$ are called the $N$-th \emph{cyclotomic classes} of $\F_q$. Furthermore, $N$-th cyclotomic classes are \emph{semi-primitive} if $-1$ is a power of $p$ modulo $N$. The next definition contains the central object of our paper. We will see that Paley graphs can be obtained from the definition below as a special case, namely $P_q=PP(q,2,\{0\})$. 

\begin{defn}\label{def:pp}
Suppose $q$ is a prime power, $d$ a positive integer such that $2d \mid (q-1)$, and $I=\{m_1, \ldots, m_d\} \subset\{0, 1, \ldots, 2d-1\}$ with $|I|=d$. Let $C_0, C_1, \ldots, C_{2d-1}$ be the $2d$-th cyclotomic classes of $\F_q$. The graph $PP(q,2d,I)$ is defined to be the Cayley graph $\operatorname{Cay}(\F_{q}^+, D)$ where 
\begin{equation}\label{D}
D=\bigcup_{j=1}^{d} C_{m_j}.
\end{equation}
Furthermore, we say that $X=PP(q, 2d, I)$ is a \emph{semi-primitive pseudo-Paley graph} if the cyclotomic classes building $X$ are semi-primitive (or equivalently, $-1$ is a power of $p$ modulo $2d$). In this case, we always write $q=p^{2rt}$ where $t$ is the smallest integer satisfying $p^{t}\equiv -1\pmod{2d}$.
\end{defn}

Note that in the above definition, $C_j$ clearly depends on the choice of $g$; however, the isomorphism class of the resulting graph will only depend on $q$, $d$, and the set $I$. Indeed, any two primitive roots, $g$ and $g'$, can be interchanged by a unique field automorphism, which gives rise to a graph isomorphism. For a similar reason, we can assume $0 \in I$ without loss of generality.

The notation \textbf{PP} stands for \textbf{P}seudo-\textbf{P}aley. Indeed, $X=PP(q, 2d, I)$ is a pseudo-Paley graph when $-1$ is a power of $p$ modulo $2d$; this immediately follows from \cite{BWX}*{Theorem 2}. We refer to Example~\ref{ex:pp} for other pseudo-Paley graphs of particular interest, including Peisert graphs and generalized Peisert graphs.

One important open problem in additive combinatorics is obtaining good estimates on the clique number of Paley graphs and similar graphs \cite{CL07}*{Section 2.7}. Recall that in a graph $X$, the clique number of $X$, denoted $\omega(X)$, is the size of a maximum clique. Since $P_q$ is strongly regular, the Delsarte bound \cite{Del73} implies that $\omega(P_q) \leq \sqrt{q}$; in fact, the same upper bound applies to any pseudo-Paley graph of order $q$. This square root upper bound is known as the \emph{trivial upper bound} on the clique number of $P_q$, and it is notoriously difficult to improve this trivial upper bound when $q$ is a non-square. Recently, a minor improvement on the trivial upper bound (from $\sqrt{q}$ to $\sqrt{q}-1$) has been made by Greaves and Soicher \cite{GS} for infinitely many pseudo-Paley graphs with non-square order. More recently, Hanson and Petridis \cite{HP} and Yip \cite{Yip1} improved the $\sqrt{q}$ bound to $\sqrt{q/2}(1+o(1))$ for $\omega(P_q)$ when $q$ is a non-square using the polynomial method. This is still far from the conjectural bound; see the related discussion in \cite{Yip1}*{Section 1} and \cite{YipG}*{Section 1.3}. Nevertheless, it is known that $\omega(P_q)=\sqrt{q}$ when $q$ is a square since the subfield $\F_{\sqrt{q}}$ forms a clique; in fact, Blokhuis \cite{Blo84} showed all maximum cliques have a subfield structure.  

\begin{thm}[\cite{Blo84}] \label{EKRPaley}
Let $q=p^{2r}$, where $p$ is an odd prime. Then the only maximum clique in the Paley graph $P_q$ containing $\{0,1\}$ is the subfield $\F_{\sqrt{q}}$. Consequently, each maximum clique is the image of an affine transformation on the subfield $\F_{\sqrt{q}}$.
\end{thm}

Let us consider the same question for a semi-primitive pseudo-Paley graph $X=PP(q,2d,I)$ constructed above. Without loss of generality, we assume that $0 \in I$. As always, we write $q=p^{2rt}$, where $t$ is the smallest integer satisfying $p^{t}\equiv -1\pmod{2d}$. Again, the Delsarte bound implies that the trivial upper bound $\omega(X) \leq \sqrt{q}$ holds. We continue the discussion on $\omega(X)$ according to the parity of $r$. 

We first consider the case when $r$ is odd. Note that in this case, $\F_{\sqrt{q}}^*$ is contained in the cyclotomic class $C_0=\langle g^{2d} \rangle$ because $\sqrt{q}=p^{rt} \equiv -1 \pmod {2d}$; therefore, $\F_{\sqrt{q}}$ forms a clique in $X$ and $\omega(X)=\sqrt{q}$. Indeed, in this case, the connection set $D$ of $X$ can be written as a union of $\F_{\sqrt{q}}^*$-cosets in $\F_q^*$, and thus, by definition $X$ is a \emph{Peisert-type graph}; see \cite{AY}*{Definition 1.1}. The structure of maximum cliques in a Peisert-type graph has been studied recently in \cites{AY, AGLY22}. In particular, each maximum clique in $X$ has a subspace structure \cite{AY}*{Theorem 1.2}. 

\begin{thm}[{\cite{AY}*{Theorem 1.2}}]\label{oddr}
Let $X=PP(q, 2d, I)$ be a semi-primitive pseudo-Paley graph with $q=p^{2rt}$ and odd $r$. Each maximum clique in $X$ containing 0 is an $\F_p$-subspace of $\F_{q}$. 
\end{thm}

There are examples for which $X$ admits a maximum clique with no subfield structure; see, for example, \cite{AY}*{Example 2.21} and \cite{AGLY22}*{Section 5}.
However, under extra assumptions, the \emph{subspace} structure of maximum cliques can be upgraded to the \emph{subfield} structure; in this direction, the authors managed to obtain an analog of Theorem~\ref{EKRPaley} in \cite{AY}*{Theorem 1.3}. See also \cite{Y24}.

In this paper, we shall focus on the case when $r$ is even. This change in parity introduces a huge difference: there is no obvious choice of a clique of size $\sqrt{q}$ in general since $\F_{\sqrt{q}}$ fails to be a clique in $X$ in general. Additionally, Lemma~\ref{lem:naive} shows how a naive analog of $\F_{\sqrt{q}}$ fails to be a valid clique. In particular, when $X=PP(p^{2r},4, \{0,1\})$ (a Peisert graph, see Section~\ref{subsect: Paley}) with $p \equiv 3 \pmod 4$ and $r$ even, it is conjectured that $\omega(X)\leq \sqrt{q}-1$ by Yip \cite{YipMaximal}; see the related discussions by Kisielewicz and Peisert \cite{KP}*{Theorem 5.1}, and Mullin \cite{NM}*{Section 3.4}. Note that a simple (yet crucial) fact in all recent improvements of the clique number of Paley graphs and related graphs (of non-square order $q$) is that $\sqrt{q}$ is not an integer when $q$ is a non-square \cites{GS, HP, Yip1, Y23}, so that the small gap between $\sqrt{q}$ and its integer part can be ``blown up" to gain an improvement on the trivial upper bound. In our case, $q$ is a square, which prevents all such techniques from being applied. To our best knowledge, there is no single infinite family of semi-primitive pseudo-Paley graphs (of square order) for which the trivial upper bound has been improved. 

The above discussion, together with Theorem~\ref{oddr},
leads to the following conjecture:
\begin{conj} \label{mainconj} Let $X=PP(q, 2d, I)$ be a semi-primitive pseudo-Paley graph with $q=p^{2rt}$ where $r$ is even. If $\omega(X)=\sqrt{q}$, then every maximum clique in $X$ is an $\F_{p^t}$-affine subspace. 
\end{conj}

We refer to Section~\ref{subsect:further-motivation} for the further motivation behind this conjecture. Conjecture~\ref{mainconj} is computationally helpful in understanding the clique number: it provides a polynomial-time algorithm to check if the clique number is given by $\sqrt{q}$. According to Conjecture~\ref{mainconj}, the only candidates for maximum cliques of size $\sqrt{q}$ are $\F_{p^t}$-subspaces of a fixed dimension, whose number is bounded by a polynomial function of $q$. We are thus led to investigate the existence of cliques with a subspace structure and with size $\sqrt{q}$ in $X$.

\begin{problem}\label{prob:subspace}
Suppose that $X=PP(q,2d,I)$ is a semi-primitive pseudo-Paley graph with $q=p^{2rt}$ where $r$ is even. Determine if $X$ has a clique of size $\sqrt{q}$ with a subspace structure. Does $X$ have a clique of size $\sqrt{q}$ with a subspace structure?
\end{problem}

We will show that the answer to the question in Problem~\ref{prob:subspace} is negative in a probabilistic sense (see Proposition~\ref{prop:density-zero}), thus lending some evidence to the inequality $\omega(X)\leq \sqrt{q}-1$ in general. However, numerical experiments suggest that there is an infinite family of semi-primitive pseudo-Paley graphs (other than Paley graphs) where the trivial upper bound is tight; see Conjecture~\ref{conj:josh}.

\subsection{Main results}\label{subsect:mainresults}

Our first main result guarantees that a maximum clique of size $\sqrt{q}$ has an equal contribution from each cyclotomic class building the graph. 

\begin{thm}\label{theorem:main-result}
Let $X=PP(q, 2d, I)$ be a semi-primitive pseudo-Paley graph with $q=p^{2rt}$ where $r$ is even. Then $\omega(X)\leq \sqrt{q}$. Moreover, if $\omega(X)=\sqrt{q}$, and $A$ is a maximum clique in $X$ such that $0 \in A$, then 
\begin{equation}\label{main}
|A \cap C_{m_1}|=|A \cap C_{m_2}|=\cdots=|A \cap C_{m_d}|=\frac{\sqrt{q}-1}{d}.    
\end{equation}
\end{thm}

When $X$ is a Paley graph, Theorem~\ref{theorem:main-result} follows from Theorem~\ref{EKRPaley}; see Lemma~\ref{lem:Paley-equal-contribution}. Furthermore, Theorem~\ref{theorem:main-result} is a genuine extension: most graphs $PP(q, 2d, I)$ are not isomorphic to the Paley graph $P_q$ by Proposition~\ref{prop:isomorphism}. This latter fact also helps us to prove a conjecture by Mullin on generalized Peisert graphs (see Corollary~\ref{cor:mullin-conjecture}).

It is crucial to assume that $r$ is even in the hypothesis of Theorem~\ref{theorem:main-result}. Indeed, when $r$ is odd, we have mentioned that $\F_{\sqrt{q}} \subset C_0 \cup \{0\}$ forms a clique in $X$, and yet the contribution solely comes from a single cyclotomic class. On the other hand, when $r$ is even, any clique of size $\sqrt{q}$ has an equal contribution from each cyclotomic class according to our main theorem. Again, we see that cliques behave very differently when $r$ is odd and when $r$ is even.

It is worth pointing out that neither Conjecture~\ref{mainconj} nor Theorem~\ref{theorem:main-result} implies each other. However, let us discuss how Theorem~\ref{theorem:main-result} shows that a maximum clique containing $0$ behaves like a subspace, thus lending some evidence towards Conjecture~\ref{mainconj}. Indeed, if $V \subset \F_q$ is a vector space over $\F_{p^t}$, we expect that the contribution to $V$ from each cyclotomic class forming the connection set $D$ would be \emph{approximately} equal. This is due to the general principle that the intersection of a subspace and a multiplicative subgroup can be modeled as the intersection of two random sets \cite{Kop19}*{Lemma 11}. On the other hand, Theorem~\ref{theorem:main-result} asserts that any maximum clique $A$ in $X$ has \emph{exactly} equal contribution from each cyclotomic class, so Conjecture~\ref{mainconj} alone does not imply Theorem~\ref{theorem:main-result}.

Assuming Conjecture~\ref{mainconj}, we apply Theorem~\ref{theorem:main-result} to improve the upper bound on the clique number of most semi-primitive pseudo-Paley graphs (unconditionally, we prove that there is no clique of size $\sqrt{q}$ with a subspace structure for most graphs). First, we show that the density of such graphs with clique number $\sqrt{q}$ is zero in Proposition~\ref{prop:density-zero}. Next, we prove an effective version of this statement. The condition that $p$ is sufficiently large in the result below is necessary; see Table~\ref{tab:evidence-main-conj} for counterexamples.

\begin{thm} \label{thm:clique_number}
Let $X=PP(q, 2d, I)$ be a semi-primitive pseudo-Paley graph with $q=p^{2rt}$ where $r$ is even, and $I\neq \{0, 2, \ldots, 2d-2\}$ and $I\neq \{1, 3, \ldots, 2d-1\}$. Assuming Conjecture~\ref{mainconj}, $\omega(X)\leq \sqrt{q}-1$ for $p^t>10.2r^2d$.
\end{thm}

As a concrete application, we can conditionally improve the Delsarte bound for the clique number of certain generalized Peisert graphs $GP^*(q,2d) \colonequals PP(q,2d,\{0,1, \ldots, d-1\})$. We remark that an unconditional improvement (even by $1$) on the Delsarte bound for the clique number of these graphs is regarded by the experts to be notoriously hard; see, for example, the discussion in the recent paper \cite{BGLR23}.

\begin{thm} \label{thm: Peisert_clique_number}
Suppose that Conjecture \ref{mainconj} is true. If $d\geq 2$, and $t$ is the smallest integer such that $p^t \equiv -1 \pmod {2d}$ and $p^t>3$, then $\omega(GP^*(p^{4t},2d))\leq p^{2t}-1$. 
\end{thm}

We also prove the following two theorems on the number of maximum cliques. Comparing these with Theorem~\ref{EKRPaley}, we see that pseudo-Paley graphs generally behave much differently compared to Paley graphs. Note that Theorem~\ref{thm:at-least-two-max-cliques} generalizes Mullin's theorem~\cite{NM}*{Lemma 3.3.6} for the Peisert graph.

\begin{thm}\label{thm:at-least-two-max-cliques}
Let $X=PP(q, 2d, I)$ be a semi-primitive pseudo-Paley graph with $q=p^{2rt}$ where $r$ is even. Suppose that $0\in I$ and $I\neq \{0, 2, \ldots, 2d-2\}$. If $\omega(X)=\sqrt{q}$, then there are at least 2 maximum cliques in $X$ that contain $\{0,1\}$.
\end{thm}

\begin{thm}\label{theorem:even-number-cliques}
Let $X=PP(q, 2d, I)$ be a semi-primitive pseudo-Paley graph with $q=p^{4t}$. Suppose that $0\in I$ and $I\neq \{0, 2, \ldots, 2d-2\}$. Then $X$ has an even number of cliques (possibly zero) of size $\sqrt{q}$ that are subspaces defined over $\F_{p^t}$ containing $1$. Furthermore, these maximum cliques come in pairs interchanged by the map $x\mapsto x^{\sqrt{q}}$.
\end{thm}

\subsection{Structure of the paper} We describe the flowchart of the paper. In Section~\ref{sect:background}, we introduce the relevant notations, provide further background, and present preliminary properties. We present the proof of Theorem~\ref{theorem:main-result} and  Theorem~\ref{thm:at-least-two-max-cliques} in Section~\ref{sect:proof-of-main-result}. In Section~\ref{sect:isomorphism}, we prove that the families of graphs considered in this paper are almost always not isomorphic to Peisert graphs or Paley graphs. In Section~\ref{sect:subspace-structure}, we prove Theorem~\ref{thm:clique_number}, Theorem~\ref{thm: Peisert_clique_number}, and Theorem~\ref{theorem:even-number-cliques}. In Appendix~\ref{sect:gauss-sums}, we present all the necessary details on Gauss sums and give a detailed proof of Proposition~\ref{proposition:main-result} in Appendix~\ref{sect:proof-of-main-prop}. Finally, in Appendix~\ref{sect:numerical-evidence}, we report results from our numerical experiments, including two algorithms. 

\section{Background and overview of the paper}\label{sect:background}

\subsection{Peisert graphs and generalized Peisert graphs}\label{subsect: Paley}


We first recall the definitions of Peisert graphs \cite{WP2}. The {\em Peisert graph} of order $q=p^r$, where $p$ is a prime such that $p \equiv 3 \pmod 4$ and $r$ is even, denoted by $P^*_q$, is the Cayley graph $\operatorname{Cay}(\F_{q}^{+}, M_q)$ with $M_q \colonequals \{g^j: j \equiv 0,1 \pmod 4\},$ where $g$ is a primitive root of the field $\F_q$. We will discuss more properties of Peisert graphs in Section~\ref{sect:isomorphism}.

Motivated by the similarity shared among Paley graphs, generalized Paley graphs, and Peisert graphs, Mullin introduced generalized Peisert graphs; see \cite{NM}*{Section 5.3} and \cite{AY}*{Definition 2.8}. Observe that $P_q=GP^*(q,2)$ and $P_q^*=GP^*(q,4)$. 

\begin{defn}[generalized Peisert graph] \label{GP*}
Let $d$ be a positive integer and $q$ a prime power such that $q \equiv 1 \pmod {4d}$. The {\em $2d$-th power Peisert graph} of order $q$, denoted $GP^*(q,2d)$, is the Cayley graph $\operatorname{Cay}(\F_q^+, M_{q,2d})$, where
$M_{q,2d}=\{g^{2dk+j}: 0\leq j \leq d-1, k \in \Z\},$
and $g$ is a primitive root of $\F_q$.
\end{defn}

Next, we will see that our Definition~\ref{def:pp} is general enough to encompass both Peisert graphs and generalized Peisert graphs. In the following discussion, when we are talking about a cyclotomic class $C_j$ in $PP(q,2d,I)$, we always assume $C_j$ is a $2d$-th cyclotomic class of $\F_q$.

\begin{ex}\label{ex:pp} By specializing the choice of the set $I$, we recover several important families of Cayley graphs we mentioned previously:
\begin{enumerate}
    \item \label{ex:pp-1} $PP(q,2,\{0\})$ is the Paley graph $P_q=\operatorname{Cay}(\F_q^{+}, C_0)$. 
    \item \label{ex:pp-2} $PP(q, 4, \{0, 1\})$ is the Peisert graph $P^{\ast}_{q} = \operatorname{Cay}(\F_q^{+}, C_0\cup C_1)$.
    \item \label{ex:pp-3} $PP(q,2d,\{0, 2, 4, \ldots,2d-2\})$ is also the Paley graph $P_q=\operatorname{Cay}(\F_{q}^{+}, C_0\cup C_2\cup \ldots \cup C_{2d-2})$.
     \item \label{ex:pp-4} $PP(q,2d,\{1, 3, 5, \ldots,2d-1\})$ is the complement of the Paley graph $P_q$, and it is isomorphic to $P_q$.
    \item \label{ex:pp-5} $PP(q,2d,\{0, 1, 2, \ldots,d-1\})$ is the generalized Peisert graph $GP^{\ast}(q, 2d)=\operatorname{Cay}(\F_q^+, C_0 \cup C_1 \cup \ldots \cup C_{d-1})$. 
    \item \label{ex:pp-6} Suppose that $d=2m$ is even. If $I=\{i, i+4, \ldots, i+4(m-1)\}\cup \{j, j+4, \ldots, j+4(m-1)\}$ for some $i<j\in \{0, 1, 2, 3\}$, then $PP(q,2d, I)$ is isomorphic to the Paley graph $P_q$ when $j-i=2$, and $PP(q,2d, I)$ is isomorphic to the Peisert graph $P_{q}^{\ast}$ when $j-i \in \{1, 3\}$.
\end{enumerate}
\end{ex}

From Example~\ref{ex:pp}, we see that two different choices of $(d, I)$ may represent the same graph. In general, if $d'\mid d$ and $I=\bigcup_{j\in I'} \{j + 2d'k :  0\leq k\leq \frac{d}{d'}-1 \}$, then $PP(q, 2d, I)=PP(q, 2d', I')$. It is natural to work with the representation with the smallest value of $d$. This motivates the following definition.

\begin{defn}[Minimal representation] We say $X=PP(q,2d,I)$ is in its {\em minimal representation} if there is no $d' \mid d$ and $I' \subset \{0,1,\ldots, 2d'-1\}$ such that $d'<d$ and $X=PP(q,2d', I')$. Note that the minimal representation of $X$ is unique. 
\end{defn}

Define $I+k \colonequals \{j+k \pmod {2d}: j \in I\}$. The next lemma gives a necessary and sufficient condition for the equality $I=I+k$ to hold.

\begin{lem}\label{lem:I=I+k}
Let $X=PP(q,2d,I)$ and let $PP(q,2d', I')$ be its minimal representation. 
Then $I=I+k$ if and only if $k$ is a multiple of $2d'$.
\end{lem}

\begin{proof}
Since $d'\mid d$ and $I=\bigcup_{j\in I'} \{j + 2d'r :  0\leq r\leq \frac{d}{d'}-1 \}$, it follows that $I=I+k$ for any $k$ that is a multiple of $2d'$.

If $I=I+k$, then $I$ is a union of arithmetic progressions, each with common difference $\gcd(2d,k)$. The length of each arithmetic progression appearing in the union is $\frac{2d}{\gcd(2d, k)}$. Therefore, $d=|I|$ is a multiple of $\frac{2d}{\gcd(2d, k)}$. It follows that $\gcd(2d, k)$ is even, that is, $\gcd(2d, k)=2e$ for some $e\in\mathbb{N}$. In particular, $X=PP(q,2d,I)=PP(q,2e,J)$, where $J=\{j \pmod {2e}: j \in I\}$. Since $PP(q,2d',I')$ is the minimal representation of $X$, it follows that $2d' \mid 2e$, and so $2d' \mid k$. \end{proof}

\begin{rem}\label{rmk:I=I+k}
The condition $I=I+k$ is connected to the symmetries of the graph $X=PP(q,2d,I)$. More precisely, the map $x\mapsto g^k x$ induces a graph automorphism of $X$ if and only if $I+k=I$. 
\end{rem} 

\subsection{Semi-primitive pseudo-Paley graph}

One of our aims in the present paper is to draw parallels between Paley graphs and the graphs of the form $PP(q, 2d, I)$. One special property Paley graphs enjoy is the following ``equal contribution" statement. Our main result, Theorem~\ref{theorem:main-result}, extends this property to a larger family of pseudo-Paley graphs.  

\begin{lem}\label{lem:Paley-equal-contribution}
Let $X=PP(q, 2d, I)$ be a semi-primitive pseudo-Paley graph with $q=p^{2rt}$ where $r$ is even and $I=\{0, 2, \ldots, 2d-2\}$, so that $X$ is the Paley graph $P_q$. If $A$ is a maximum clique in $X$ such that $0 \in A$, then $$|A \cap C_{0}|=|A \cap C_{2}|=\cdots=|A \cap C_{2d-2}|=\frac{\sqrt{q}-1}{d}.$$
\end{lem}

\begin{proof}
Note that $g^{p^{rt}+1}$ is a primitive root of the subfield $\F_{\sqrt{q}}$. Since $p^t \equiv -1 \pmod {2d}$ and $r$ is even, $p^{rt}+1 \equiv 2 \pmod {2d}$.  
Let $A$ be a maximum clique in $X$ such that $0 \in A$. Since $0 \in A$, by Theorem \ref{EKRPaley}, $A=a\F_{\sqrt{q}}$ for some square $a \in \F_q^*$.  Observe that for each $0\leq j\leq d-1$,
\begin{equation}\label{subfield}
\F_{\sqrt{q}}\cap C_{2j} = \{ (g^{p^{rt}+1})^k \ : \ k \equiv j \pmod d\ \} \ \ \Rightarrow \ \ |\F_{\sqrt{q}}\cap C_{2j}|=\frac{\sqrt{q}-1}{d}.
\end{equation}
The same conclusion holds for $A=a \F_{\sqrt{q}}$. 
\end{proof}

From Lemma~\ref{lem:Paley-equal-contribution}, we see that the subfield $\F_{\sqrt{q}}$ admits the following presentation
$$
\F_{\sqrt{q}}=\{0\} \cup \bigcup_{j=0}^{d-1} \{ (g^{p^{rt}+1})^k \ : \ k \equiv j \pmod d\ \}=\{0\} \cup \bigcup_{j=0}^{d-1} g^{2j} \langle g^{d(\sqrt{q}+1)}\rangle.
$$
This motivates the following construction of a set in $PP(q,2d,I)$:
\begin{equation}\label{naive_construction}
A(q,2d,I)=\{0\} \cup \bigcup_{j=1}^{d} g^{m_j} \langle g^{d(\sqrt{q}+1)}\rangle.
\end{equation}
The following lemma shows this ``naive construction" of $A(q,2d,I)$ does not give a clique in $PP(q,2d,I)$ except for the Paley graph.

\begin{lem}\label{lem:naive}
Let $X=PP(q, 2d, I)$ be a semi-primitive pseudo-Paley graph with $q=p^{2rt}$ where $r$ is even. Then $A(q,2d,I)$ defined in equation~\eqref{naive_construction}
is a clique in $X$ if and only if $I = \{0,2,\ldots, 2d-2\}$ or $I=\{1,3, \ldots, 2d-1\}$.
\end{lem}

\begin{proof}
If $I=\{0,2,\ldots, 2d-2\}$, then the conclusion follows from Lemma~\ref{lem:Paley-equal-contribution}. The case $I=\{1,3, \ldots, 2d-1\}$ is similar. 

Conversely, assume that $I$ is different from $\{0, 2, \ldots, 2d-2\}$ and $\{1, 3, \ldots, 2d-1\}$. Given $X=PP(q,2d,I)$, without loss of generality, we may assume that $m_1=0$, and that $X$ is already in its minimal representation: indeed, if $PP(q, 2d, I)=PP(q,2d', I')$, then $A(q,2d',I')=A(q,2d,I)$. Thus, by hypothesis, $d\geq 2$.

Assume, to the contrary, that $A(q,2d,I)$ is a clique in $X$. Let $B_j=A(q,2d,I) \cap (C_j \cup \{0\})$ for each $0 \leq j \leq 2d-1$. If $x,y \in B_0$ such that $x-y \in C_k$ with $0\leq k\leq 2d-1$, then for each $1 \leq j \leq d$, we have $g^{m_j}x, g^{m_j}y \in B_{m_j}$ and $g^{m_j}x-g^{m_j}y \in C_{k+m_j}$. As a result, $I+k=\{m_j+k \pmod {2d}: 1 \leq j \leq d\}=I$. By Lemma~\ref{lem:I=I+k}, $2d|k$ as $PP(q,2d,I)$ is the minimal representation of $X$. Thus, $k=0$ and $B_0-B_0 \subset B_0$.

From $B_0-B_0\subset B_0$ we infer that $B_0$ is a subgroup of $\F_{q}^{+}$, and must have cardinality $|B_0|=p^n$ for some $n\geq 1$. We show this is impossible by analyzing the following Diophantine equation: 
$$
\frac{\sqrt{q}-1}{d}+1 = |B_0| = p^n.
$$
Writing $\sqrt{q}=p^{rt}$, and rearranging the equation, we obtain $p^{rt}-1 = d(p^n-1)$. In particular, $d\equiv 1 \pmod{p^n}$, implying that $d \geq p^n+1$. Combining $d\geq p^n+1$ and $p^{rt}-1=d(p^n-1)$, we obtain $rt\geq 2n$. Thus, $d \geq p^{rt/2}+1$. Since $p^t\equiv -1\pmod{2d}$, we have $p^t\geq 2d-1>2p^{rt/2} \geq 2p^t$, a contradiction. This proves that $A(q, 2d, I)$ is not a clique in $X$.
\end{proof}

We will later see in the proof of Theorem~\ref{thm:at-least-two-max-cliques} how Lemma~\ref{lem:naive} can help us distinguish $X$ from the Paley graph. Indeed, we will show that most semi-primitive pseudo-Paley graphs are not isomorphic to the Paley graph; see Corollary~\ref{cor:non-isomorphism-semi-primitive} for a precise criterion.

\subsection{Further motivation}\label{subsect:further-motivation}
Variants of Erd\H{o}s-{K}o-{R}ado (EKR) theorem have been studied intensively; the book by Godsil and Meagher~\cite{GM} provides an excellent survey. Typically, EKR-type results state that the extremal configurations must be canonical in the sense that they have a remarkably simple structure compared to the complicated ground space. Theorem~\ref{EKRPaley} can be regarded as the EKR theorem of Paley graphs \cite{GM}*{Section 5.9}: if we view cliques with a subfield structure as canonical cliques, then Theorem~\ref{EKRPaley} states that all maximum cliques are canonical.

In our setting, $X=PP(q,2d,I)$ is a Cayley graph $\operatorname{Cay}(\F_q^+, D)$, so $A\subset\F_q$ forms a clique in $X$ if and only if $A-A \subset D \cup \{0\}$, where $A-A=\{a-a': a,a' \in A\}$. We see that additive combinatorics plays an important role here because $A-A$ has a rich additive structure, whereas $D$, being a union of cyclotomic classes, has a rich multiplicative structure. This partially explains why estimating the clique number of Paley graphs and other related graphs is a central open problem in additive combinatorics \cite{CL07}*{Section 2.7}: the study of a clique is essentially the study of the interaction between additive structure and multiplicative structure of a set. All results in this paper conditional on Conjecture~\ref{mainconj} can be reformulated unconditionally based on such interaction; for instance, compare Theorem~\ref{thm:clique_number} and Corollary~\ref{cor:clique-number-additive}.

Let $K$ be a subfield of $\F_q$. 
Note that if $V$ is a $K$-subspace, then $V-V=V$; thus, the set $V$ forms a clique in $X$ if and only if $V \subset D \cup \{0\}$. Therefore, it is much easier for a subspace to form a clique in $X$ compared to a generic subset of $\F_q$. We would like to choose the subfield $K$ to serve the role of the prime subfield $\F_{p}$. In our case $q=p^{2rt}=(p^t)^{2r}$ and $t$ is the smallest integer with $p^t\equiv -1\pmod{2d}$. Given Stickelberger's Theorem (Theorem~\ref{forr}), the choice $K=\F_{p^t}$ is a natural replacement of the prime subfield $\F_{p}$. Therefore, we call $V$ a {\em canonical clique} in $X$ if $V$ is a $\F_{p^t}$-affine subspace with $|V|=\sqrt{q}$ such that $V$ forms a clique in $X$. In contrast to a typical EKR-type result, canonical objects may not exist in our situation because $X$ may have a clique number less than $\sqrt{q}$. Using this terminology, Conjecture~\ref{mainconj} states that each clique of size $\sqrt{q}$ in $X$ is canonical. 

We remark that our conjecture is an analog of the celebrated Chv\'atal's conjecture \cite{Ch72} on families of set systems,  which is a variant of the EKR theorem and is still widely open (see Chv\'atal's webpage \cite{Chweb} for related discussion and references on the conjecture). Chv\'atal's conjecture states that for any family $\mathcal{F}$ of subsets of a finite set that is closed under taking subsets (known as an \emph{ideal}), there is a largest intersecting subfamily of $\mathcal{F}$ that is a star. In other words, in any nice set system, there is a canonical maximum intersecting family; our conjecture is actually stronger because it implies that in any nice semi-primitive pseudo-Paley graph (that is, with clique number $\sqrt{q}$), every maximum clique is canonical. In summary, Conjecture~\ref{mainconj} can be viewed as a combination of the EKR theorem and Chv\'atal's conjecture in the context of pseudo-Paley graphs. The following table is a short dictionary between concepts in Chv\'atal's conjecture (first column) and the present work (second column).

\begin{center}
\begin{tabular}{ |c|c| } 
 \hline
  uniform set system & Paley graph \\ 
 \hline
 set system & semi-primitive pseudo-Paley graph \\
 \hline
 nice set system & nice semi-primitive pseudo-Paley graph  \\ 
 (ideal, i.e., closed under taking subsets) & (with clique number $\sqrt{q}$) \\
  \hline 
 largest intersecting family & maximum clique\\
 \hline
 canonical intersecting family (star) & canonical clique (affine subspace) \\ 
  \hline
  The original EKR theorem & Blokhuis theorem (Theorem~\ref{EKRPaley}) \\ 
 \hline
  Chv\'atal's conjecture + EKR & Conjecture~\ref{mainconj} \\ 
 \hline
\end{tabular}
\end{center}

\section{Proof of Theorem \ref{theorem:main-result}}\label{sect:proof-of-main-result}
The novel idea in the proof of Theorem \ref{theorem:main-result} is to establish a system of linear equations in the expressions $|A \cap C_{m_1}|, |A \cap C_{m_2}|, \ldots, |A \cap C_{m_d}|$ and then find a way to solve the system. A key ingredient is the following Fourier analytic characterization of maximum cliques in a semi-primitive pseudo-Paley graph, which is of independent interest. We need the following definition of exponential sums (scaled Fourier coefficients of the indicator function on $A$). Recall that $e_p(x)=\exp(2 \pi i x/p)$ for all $x \in \F_p$, and $\Tr_{\F_q}(\cdot)\colon\F_q \to \F_p$ is the absolute trace map.

\begin{defn} \label{def:expsum}
Let $A$ be a subset of $\F_q$. For each $c \in \F_q^*$, we define
\begin{equation}\label{Sqac}
 S(q,A;c)=\sum_{a\in A} e_p\big(\Tr_{\F_q}(ac)\big).
\end{equation}
\end{defn}

\begin{prop}\label{proposition:main-result}
Let $X=PP(q, 2d, I)$ be a semi-primitive pseudo-Paley graph with $q=p^{2rt}$. If $A$ is a clique in $X$, then $|A| \leq \sqrt{q}$; moreover, $|A|=\sqrt{q}$ if and only if $S(q,A;c)=0$ for all $c \in D'$, where $S(q,A;c)$ is defined in equation~\eqref{Sqac} and $D'=\cup_{\ell=1}^{d} C_{-m_\ell}$. 
\end{prop}

The proof of Proposition~\ref{proposition:main-result} is technical and is deferred to Appendix~\ref{sect:proof-of-main-prop}. A similar characterization for maximum cliques in generalized Paley graphs has appeared in \cite{YipG}*{Theorem 1.4}.

Now we present the proof of Theorem~\ref{theorem:main-result} using Proposition~\ref{proposition:main-result} and the novel idea mentioned at the beginning of this section.

\begin{proof}[Proof of Theorem \ref{theorem:main-result}] Let $A$ be a clique of size $\sqrt{q}$ such that $0 \in A$. Let $A_j=A \cap C_j$ for each $0\leq j\leq 2d-1$.  Since $A$ is a clique in $X$ and $0 \in A$, we have $A \setminus \{0\}=\cup_{j=1}^{d} A_{m_{j}}$, and thus
\begin{equation}\label{sumaj}
\sum_{j=1}^{d}|A_{m_j}|=\sqrt{q}-1.  
\end{equation} 

By Proposition~\ref{proposition:main-result}, $S(q,A;c)=0$ for all $c \in D'$ where $D'=C_{-m_{1}}\cup C_{-m_{2}}\cup \cdots \cup C_{-m_{d}}$. Without loss of generality, we may assume that $m_1=0$. In particular, we have $S(q,A;c)=0$ for each $c\in C_0$, implying the following:
\begin{align*}
0&= \sum_{c \in C_0} S(q, A;c)=\sum_{c \in C_0}  \sum_{a\in A} e_p\big(\Tr_{\F_q}(ac)\big)= \sum_{a\in A} \sum_{c \in C_0} e_p\big(\Tr_{\F_q}(ac)\big) \\
&=|C_0|+\sum_{j=1}^{d}\sum_{a\in A_{m_j}} \sum_{c \in C_0} e_p\big(\Tr_{\F_q}(ac)\big).  
\end{align*}

Note that if $a \in A_{m_j}$, then as $c$ runs over $C_0$, $ac$ runs over $C_{m_j}$. Therefore,
\begin{align*}
0&=|C_0|+\sum_{j=1}^{d}|A_{m_j}| \sum_{c \in C_{m_{j}}} e_p\big(\Tr_{\F_q}(c)\big).
\end{align*}

We can apply the similar argument to $C_{-m_{k}}$ for each $k\geq 2$ to obtain:
\begin{align*}
0&= \sum_{c \in C_{-m_{k}}} S(q, A;c)= \sum_{a\in A} \sum_{c \in C_{-m_k}} e_p\big(\Tr_{\F_q}(ac)\big) \\
&=|C_{-m_{k}}|+\sum_{j=1}^{d}|A_{m_{j}}| \sum_{c \in C_{m_j-m_k}} e_p\big(\Tr_{\F_q}(c)\big).  
\end{align*}

Recall that the index $j$ in $C_{j}$ is defined up to modulo $2d$, which allows us to consider $C_{m_j-m_k}$ even when $m_j-m_k<0$. For convenience, we define
$$
\lambda \colonequals \sum_{c\in C_0} e_p\big(\Tr_{\F_q}(c)\big), \quad \text{and} \quad \mu\colonequals \sum_{c\in C_k} e_p\big(\Tr_{\F_q}(c)\big) 
$$
for each $1\leq k \leq 2d-1$. The value of $\mu$ is well-defined in view of Corollary~\ref{subsum}.

Given a fixed integer $k$ satisfying $1\leq k\leq d$, there is a unique value of $j$ such that $m_j-m_k\equiv 0 \pmod{2d}$, namely, $j=k$. Therefore, the equation
$$
0=|C_{-m_{k}}|+\sum_{j=1}^{d}|A_{m_{j}}| \sum_{c \in C_{m_j-m_k}} e_p\big(\Tr_{\F_q}(c)\big) 
$$
becomes
\begin{equation} \label{eq: lambdamu}
0=|C_0|+ \lambda |A_{m_k}| +  \sum_{\substack{1 \leq j \leq d \\ j\neq k}} \mu |A_{m_j}| 
\end{equation}
where we also used the fact that $|C_{-m_k}|=|C_0|$. Note that equation~\eqref{eq: lambdamu} is valid for each $1\leq k\leq d$, and so these $d$ linear equations correspond to the matrix equation $\mathbf{M} \mathbf{x} = \mathbf{b}$ where
$$
\mathbf{M} = \begin{pmatrix}
\lambda & \mu &\mu &\hdots & \mu \\
\mu & \lambda &\mu &\hdots & \mu \\
\mu & \mu &\lambda &\hdots & \mu \\
\vdots & \vdots &\vdots &\ddots & \vdots \\
\mu & \mu &\mu &\hdots & \lambda \\
\end{pmatrix}, \ \ \ \mathbf{x}=\begin{pmatrix}
|A_{m_{1}}| \\
|A_{m_{2}}| \\
|A_{m_{3}}| \\
\vdots \\
|A_{m_{d}}|
\end{pmatrix} \text{ and } \ \mathbf{b} = - \begin{pmatrix} |C_0| \\ |C_0| \\ |C_0| \\ \vdots \\ |C_0| \end{pmatrix}
$$
Note that $\mathbf{M}=(\lambda-\mu)I_{d} + \mu J_{d}$, where $I_{d}$ is the $d\times d$ identity matrix and $J_{d}$ is the $d\times d$ all $1$'s matrix. It is straightforward to see that $\mathbf{M}$ has eigenvalue $\lambda-\mu$ with multiplicity $d-1$, and eigenvalue $\lambda+(d-1)\mu$ with multiplicity $1$. By Corollary~\ref{subsum}, we know that $\lambda\neq \mu$ and 
$$
\lambda+(d-1)\mu = -\frac{(2d-1)\sqrt{q}+1}{2d}+ (d-1)\cdot \frac{\sqrt{q}-1}{2d} = -\frac{\sqrt{q}+1}{2}.
$$
Therefore, $\mathbf{M}$ is invertible. Thus, $\mathbf{M}\mathbf{x}=\mathbf{b}$ has a unique solution given by
$$
\mathbf{x} = \begin{pmatrix} 
\gamma \\ 
\gamma \\
\vdots \\
\gamma
\end{pmatrix}
$$
where $\gamma = -|C_0|/(\lambda+(d-1) \mu)$. We deduce that 
$$
|A_{m_j}| = \gamma=-\frac{|C_0|}{\lambda+(d-1) \mu}=\frac{2|C_0|}{\sqrt{q}+1} = \frac{2(q-1)}{2d(\sqrt{q}+1)} = \frac{\sqrt{q}-1}{d} 
$$
for each $1\leq j\leq d$.  It is easy to verify that $|A_{m_1}|,|A_{m_2}|, \ldots, |A_{m_d}|$ satisfy the equation \eqref{sumaj}. \end{proof}

As a quick application, Theorem~\ref{theorem:main-result} immediately implies the following nontrivial upper bound on the clique number of a particular type of Cayley graphs.

\begin{cor}\label{cor:main-result}
Let $X=PP(q, 2d, I)$ be a semi-primitive pseudo-Paley graph with $q=p^{2rt}$ where $r$ is even. Pick subsets $D_{m_j}\subset C_{m_j}$ for each $1\leq j\leq d$ such that $|D_{m_k}|<(\sqrt{q}-1)/d$ holds for some $k$. Consider the Cayley graph $X'=\operatorname{Cay}\left(\F_q^+, \bigcup_{j=1}^{d} D_{m_j}\right)$ , which is a subgraph of $X$. Then $\omega(X')\leq \sqrt{q}-1$. 
\end{cor}

\begin{proof}
By Theorem~\ref{theorem:main-result}, $\omega(X') \leq \omega(X) \leq \sqrt{q}$. So, if $\omega(X')= \sqrt{q}$, then $\omega(X)=\sqrt{q}$ as well. Now, take a maximum clique $A$ in $X'$ containing $0$, so that $A$ is also a maximum clique in $X$. By Theorem~\ref{theorem:main-result}, $|D_{m_j}|\geq |A \cap C_{m_j}|=\frac{\sqrt{q}-1}{d}$ for each $1\leq j\leq d$, contradicting the hypothesis. 
\end{proof}

We continue with another application of Theorem~\ref{theorem:main-result} on the number of maximum cliques. Recall that in the Paley graph $P_q$, there is a unique maximum clique that contains $\{0, 1\}$, namely the subfield $\F_{\sqrt{q}}$. In contrast, Mullin~\cite{NM}*{Lemma 3.3.6} shows that if $r$ is even and $q=p^{2r}$ for some prime $p \equiv 3 \pmod 4$ such that $\omega(P_q^{\ast})=\sqrt{q}$, then there are at least 2 maximum cliques in $P_q^{\ast}$ that contain $\{0,1\}$. Next, we present the proof of Theorem~\ref{thm:at-least-two-max-cliques}, which generalizes Mullin's result.

\begin{proof}[Proof of Theorem~\ref{thm:at-least-two-max-cliques}]
Suppose $\omega(X)=\sqrt{q}$, and that there is a unique maximum clique $A$ in $X$ that contains $\{0,1\}$. We will show that $A=A(q,2d,I)$ defined in equation~\eqref{naive_construction},
which contradicts Lemma~\ref{lem:naive} and thus establishes the theorem. 

Let $A_j=A \cap C_j$ for each $0 \leq j \leq 2d-1$. By Theorem~\ref{theorem:main-result}, $|A_{m_j}|=\frac{\sqrt{q}-1}{d}$ for each $1 \leq j \leq d$. For each $y \in A_0$, the map $x \mapsto y^{-1}x$ induces a graph automorphism of $X$, so $y^{-1}A$ is a maximum clique; moreover, $1=y^{-1}y \in y^{-1}A$, and so $y^{-1}A=A$ by the uniqueness of the maximum clique containing $\{0,1\}$. In particular, for each $x, y \in A_0$, we have $xy^{-1} \in y^{-1}A_0=A_0$. Hence $A_0$ is a subgroup of $\F_q^{\ast}$. Since $|A_0|=\frac{\sqrt{q}-1}{d}$, we conclude that $A_0=\langle g^{d(\sqrt{q}+1)}\rangle$. For each $1 \leq j \leq d$, we can apply the same argument to $g^{-m_j}A_{m_j}$ to conclude that $A_{m_j}=g^{m_j} \langle g^{d(\sqrt{q}+1)}\rangle$. Therefore,
$$
A = \{0\} \cup \bigcup_{j=1}^{d} A_{m_j} = \{0\} \cup \bigcup_{j=1}^{d} g^{m_j} \langle g^{d(\sqrt{q}+1)}\rangle =A(q,2d,I)$$
as desired. \end{proof}

\begin{cor}\label{cor:non-isomorphism-semi-primitive}
Let $X=PP(q, 2d, I)$ be a semi-primitive pseudo-Paley graph with $q=p^{2rt}$ where $r$ is even. Then $X$ is isomorphic to the Paley graph $P_q$ if and only if $I= \{0,2,\ldots, 2d-2\}$ or $I=\{1,3, \ldots, 2d-1\}$.
\end{cor}

\begin{proof}
If $I=\{0, 2, \ldots, 2d-2\}$ or $I=\{1, 3, \ldots, 2d-1\}$, then $X\cong P_q$ by Example~\ref{ex:pp}.

Conversely, suppose that $I$ is different $\{0,2,\ldots, 2d-2\}$ and $\{1,3, \ldots, 2d-1\}$. We may assume that $0\in I$. Assume, to the contrary, that $X$ is isomorphic to the Paley graph $P_q$. Then $\omega(X)=\sqrt{q}$.  Let $f\colon X\to P_q$ be a graph isomorphism. Since $0 \in I$ and $1 \in C_0$, it follows that $0$ and $1$ are adjacent in $X$, and $f(0)$ and $f(1)$ are adjacent in $P_q$. After composing $f$ with the automorphism $x \mapsto (x-f(0))/(f(1)-f(0))$ of the Paley graph $P_q$, we obtain an isomorphism $f'\colon X\to P_q$ fixing $0$ and $1$. This implies that the number of maximum cliques containing $\{0, 1\}$ must be the same in both graphs; however, there is a unique such maximum clique in $P_q$ by Theorem~\ref{EKRPaley}, while there are at least two such cliques in $X$ by Theorem~\ref{thm:at-least-two-max-cliques}. This contradiction shows that $X$ is not isomorphic to $P_q$. \end{proof}

The corollary above shows that almost all the graphs of the form $PP(q, 2d, I)$ in the semi-primitive case are not isomorphic to the Paley graph. In Section~\ref{sect:isomorphism} below, we will prove a stronger result without the semi-primitive assumption using different tools. We included the weaker statement above because its proof is self-contained. 

\section{Non-isomorphism}\label{sect:isomorphism}

The goal of this section is to show that most graphs of the form $PP(q, 2d, I)$ are not isomorphic to the Paley graph $P_q$ or the Peisert graph $P_q^{\ast}$. To achieve this result, we will borrow the deep results obtained by Peisert, Foulser, Kallaher, and others.

Peisert~\cite{WP2} proved that $P_{q}$ and $P^{\ast}_{q}$ are non-isomorphic except when $q=3^2$ by studying their respective automorphism groups. Another approach to distinguish these graphs is to study their $p$-ranks \cite{BvE92}; this was carried out successfully in \cite{WQWX}*{Proposition 3.6}. We will use ideas in Peisert's paper ~\cite{WP2} to establish our key criterion in Proposition~\ref{prop:isomorphism}.

Throughout this section, we assume that $q=p^n$ for some positive integer $n$. For any field $F$ and a positive integer $m$, we can consider the \emph{general semilinear group} $\operatorname{\Gamma L}_m(F)$ \cite{dLF11}*{Section 3.6}. We are primarily interested in the case when $F=\F_q$ and $m=1$. In this case, we often write $\operatorname{\Gamma L}_1(q)$ instead of $\operatorname{\Gamma L}_1(\F_q)$.

\begin{defn}[general semilinear group] Let $\operatorname{\Gamma L}_1(q) \colonequals \F_q^{\ast}\rtimes \operatorname{Gal}(\F_q / \F_p)$ be the semi-direct product obtained via the natural action of $\operatorname{Gal}(\F_q / \F_p)$ on $\F_q^{\ast}$. This group is called the \emph{general semilinear group} of degree $1$ over the finite field $\F_q$.
\end{defn}

Let $g$ be a fixed primitive root of $\F_q$, and let $\zeta: \F_q\to \F_q$ be the map $x\mapsto gx$. Since $\zeta^k: \F_q \to \F_q$ is given by $x\mapsto g^k x$, we see that $\langle \zeta \rangle \cong \F_q^{\ast}$. Thus, the elements of $\operatorname{\Gamma L}_1(q)$ can be viewed as functions $f\colon \F_q\to \F_q$ satisfying $f(x+y)=f(x)+f(y)$ for all $x,y\in \F_q$ and $f(cx)=c^{p^j} x$ for some $j$ and for all $x, c\in \F_q$. Since the Galois group $\operatorname{Gal}(\F_q / \F_p)=\langle \alpha \rangle$ is generated by the Frobenius map $\alpha: x\mapsto x^p$, we can also write $\operatorname{\Gamma L}_1(q)=\F_q^{\ast}\rtimes \langle \alpha \rangle = \langle \zeta \rangle \rtimes \langle \alpha \rangle$. Thus, $\operatorname{\Gamma L}_1(q) = \langle \zeta, \alpha \rangle$ and in particular, $\operatorname{\Gamma L}_1(q)$ is a permutation group. We remark that other authors (for example Peisert~\cite{WP2}, and Foulser and Kallaher~\cite{FK78}) use $\omega$ to denote the map $x\mapsto g x$. We use $\zeta$ for this purpose as $\omega$ is reserved for the clique number throughout the paper. 

Peisert \cite{WP2} classified all self-complementary symmetric graphs. Recall that a graph is self-complementary if it is isomorphic to its complement, and a graph is symmetric if its automorphism group acts transitively on the ordered pairs of adjacent vertices.

\begin{thm}[\cite{WP2}] \label{SCSgraph}
A graph $G$ is self-complementary and symmetric if and only if $|G|=p^{n}$ for some prime $p$ with $p^{n} \equiv 1\pmod{4}$, and $G$ is isomorphic to a Paley graph, a Peisert graph, or the exceptional graph $G\left(23^{2}\right)$.
\end{thm}

The key to proving Theorem~\ref{SCSgraph} is the classification of the automorphism groups of all possible self-complementary symmetric graphs \cite{WP2}*{Theorem 4.2}. For our purposes, we state a slight variation of this result. 

\begin{prop}\label{prop:automorphism-group-classification}
Let $X$ be a self-complementary symmetric Cayley graph defined on $\F_{p^n}$ with $\mathscr{A}=\operatorname{Aut}(X)$, and $\mathscr{A}_0$ be the subgroup of $\mathscr{A}$ which fixes $0$. If $p^n \notin \{9, 49, 81,529\}$, then $\mathscr{A}_0$ is a subgroup of $\operatorname{\Gamma L}_{1}(p^n)$ and $\mathscr{A} \cong \F_{p^n}\times \mathscr{A}_0 $.
\end{prop}

\begin{proof}
Each automorphism $f\in \mathscr{A}$ can be decomposed as $f_2\circ f_1$ where $f_1\in \mathscr{A}_0$ and $f_2$ is a translation by an element of $\F_{p^n}$.  This explains the identity $\mathscr{A} \cong \F_{p^n}\times \mathscr{A}_0 $.

The proof in \cite{WP2}*{Theorem 4.2} classifies the automorphism group of any such graph $X$. There are four different cases in Peisert's proof. Case 1 corresponds exactly to $\mathscr{A}_0 \leq \operatorname{\Gamma L}_1(p^n)$. Next, Case 2 corresponds to $p^n=9$, while Cases 3 and 4 correspond to $p^n\in \{49, 81, 529\}$. By the hypothesis, we must be in Case 1. 
\end{proof}

The following lemma shows that a subgroup of $\operatorname{\Gamma L}_1(p^n)$ has a simple structure. The sentence starting with ``Moreover" is not literally part of \cite{FK78}*{Lemma 2.1}, but appears on page 117 of \cite{FK78} and also on page 220 of \cite{WP2}.

\begin{lem}[\cite{FK78}*{Lemma 2.1}]\label{lem:size-of-A0} 
If $\mathscr{A}_{0}$ is a subgroup of $\operatorname{\Gamma L}_1(p^n)$, then  $\mathscr{A}_{0}=\left\langle\zeta^{d}, \zeta^{e} \alpha^{s}\right\rangle$ for some integers $d, e, s$ satisfying the following conditions:
$$
s\mid n, d \mid \left(p^{n}-1\right), e\left(\frac{p^{n}-1}{p^{s}-1}\right) \equiv 0 \pmod{d}.
$$
Such a representation $\mathscr{A}_{0}=\left\langle\zeta^{d}, \zeta^{e} \alpha^{s}\right\rangle$ is unique when $d, s>0$ and $0 \leq e<d$ and it is called the \emph{standard form}. Moreover, if $\mathscr{A}_{0}=\left\langle\zeta^{d}, \zeta^{e} \alpha^{s}\right\rangle$ is in standard form, then $|\mathscr{A}_0| =n\left(p^{n}-1\right) / s d.$
\end{lem}

The automorphism group of Paley graphs and Peisert graphs were first determined by Carlitz \cite{Car60} and Peisert \cite{WP2}, respectively.

\begin{thm}[\cite{Car60}] \label{autPaley}
Let $q=p^n$. Then $\operatorname{Aut}(P_q) \cong \F_{p^n} \times \left\langle\zeta^{2}, \alpha\right\rangle$ and $|\operatorname{Aut}(P_q)|=n p^{n}\left(p^{n}-1\right)/2$.
\end{thm}

\begin{thm}[\cite{WP2}*{Corollary 6.3}]\label{autPeisert}
Let $q=p^n$, where $p \equiv 3 \pmod 4$ and $n$ is even. If $q \neq 3^{2}, 7^{2}, 3^{4}, 23^{2}$, then $\operatorname{Aut}(P_q^*) \cong \F_{p^n} \times \left\langle\zeta^{4}, \zeta \alpha\right\rangle$ and $|\operatorname{Aut}(P_q^*)|=n p^{n}\left(p^{n}-1\right) / 4$.
\end{thm}

For a generalized Paley graph, determining its automorphism group is much more difficult. In \cite{LP}, Lim and Praeger used an association scheme to study the automorphism groups of certain generalized Paley graphs. We expect it is even harder to determine the automorphism group of $X=PP(q, 2d, I)$ since it is a union of copies of generalized Paley graphs. While determining the automorphism group of $X$ remains open, we show that $X$ is not isomorphic to $P_q$ or $P_q^{\ast}$ in general. The following proposition is a stronger version of Corollary~\ref{cor:non-isomorphism-semi-primitive}.

\begin{prop}\label{prop:isomorphism}
Let $X=PP(q, 2d, I)$ be in its minimal representation such that $q=p^n\notin\{3^2, 7^2, 3^4, 23^2\}$.
\begin{enumerate}
    \item If $d=1$, then $X$ is isomorphic to the Paley graph $P_q$. 
    \item If $d=2$, then $X$ is isomorphic to the Peisert graph $P_{q}^{\ast}$. 
    \item If $d \geq 3$, then $X$ is not isomorphic to $P_q$ or $P_q^{\ast}$.
\end{enumerate}
\end{prop}

\begin{proof}
It is straightforward to prove the cases $d=1$ and $d=2$. See Example~\ref{ex:pp}.

Next, we assume $d \geq 3$. Assume, to the contrary, that $X$ is isomorphic to $P_q$ or to $P_{q}^{\ast}$. In particular, $X$ is symmetric and self-complementary by Theorem~\ref{SCSgraph}. Let $\mathscr{A}=\operatorname{Aut}(X)$ and $\mathscr{A}_0\leq \mathscr{A}$ be the subgroup fixing $0$. By Proposition~\ref{prop:automorphism-group-classification}, $\mathscr{A}_{0}$ is a subgroup of $\operatorname{\Gamma L}_1(\F_q)$. 

By Lemma~\ref{lem:I=I+k} and Remark~\ref{rmk:I=I+k}, $\zeta^k: x\mapsto g^k x$ induces a graph automorphism of $X$ if and only $2d \mid k$. Therefore, each $\zeta^k\in \mathscr{A}_0$ with $k>0$ must satisfy $k\geq 2d \geq 6$. Using Lemma~\ref{lem:size-of-A0}, we have $\mathscr{A}_0 = \langle \zeta^{d'}, \zeta^{e} \alpha^{s} \rangle$ in its standard form for some $d'\geq 6$, $e\geq 0$ and $s\geq 1$.  Consequently, 
$$
|\mathscr{A}|=|\F_{p^n}||\mathscr{A}_0| = \frac{np^n (p^n-1)}{sd'} \leq \frac{np^n (p^n-1)}{6}.
$$
The automorphism group of $X$ has a smaller size than the automorphism group of both the Paley graph $P_q$ and the Peisert graph $P_q^*$ (if $n$ is even) shown in Theorem~\ref{autPaley} and Theorem~\ref{autPeisert}, which is a contradiction.
\end{proof}

Mullin \cite{NM}*{Chapter 8} conjectured that generalized Peisert graphs are distinct from Peisert and Paley graphs
for infinitely many prime powers. She verified the conjecture computationally for small prime powers \cite{NM}*{Section 5.4}. We confirm this conjecture below.

\begin{cor}\label{cor:mullin-conjecture}
If $d \geq 3$ and $q \equiv 1 \pmod {4d}$ and $q\notin\{3^2, 7^2, 3^4, 23^2\}$, then $GP^*(q,2d)$ is not isomorphic to $P_q$ or $P_q^*$. In fact, $GP^*(q,2d)$ is self-complementary but not symmetric.
\end{cor}

\begin{proof}
By Example~\ref{ex:pp}, we identify $GP^{\ast}(q, 2d)$ with $PP(q, 2d, \{0, 1, \ldots, d-1\})$ which is in its minimal representation, and the complement of $GP^{\ast}(q, 2d)$ corresponds to $PP(q, 2d, \{d, d+1, \ldots, 2d-1\})$. The first assertion follows directly from Proposition~\ref{prop:isomorphism}. Consider the graph homomorphism:
\begin{align*}
    f\colon PP(q, 2d, \{0, 1, \ldots, d-1\})\to PP(q, 2d, \{d, d+1, \ldots, 2d-1\})
\end{align*} induced by $x\mapsto g^d x$ on the set of vertices. Then $f$ is an isomorphism from $GP^{\ast}(q,2d)$ to its complement. Thus, $GP^{\ast}(q,2d)$ is self-complementary. Since $GP^{\ast}(q,2d)$ is not isomorphic to the Paley graph or the Peisert graph, $GP^{\ast}(q,2d)$ cannot be symmetric by Theorem~\ref{SCSgraph}.
\end{proof}

Mathon \cite{Mathon} asked whether there is an infinite family of self-complementary strongly regular graphs of non-Paley type. This was already settled affirmatively by Peisert in Theorem~\ref{SCSgraph} as Peisert graphs are of non-Paley type. Klin, Kriger, and Woldar \cite{KKW16} also answered Mathon's question for graphs of order $p^2$. Corollary~\ref{cor:mullin-conjecture} gives another answer to this question. Indeed, when $-1$ is a power of $p$ modulo $2d$, the generalized Peisert graph $GP^{\ast}(q, 2d)$ is a self-complementary strongly regular graph, but not isomorphic to $P_q$ or $P_q^{\ast}$.

\section{The subspace structure of maximum cliques}\label{sect:subspace-structure}

In previous sections, we saw that a semi-primitive pseudo-Paley graph $PP(q,2d,I)$ is not isomorphic to $P_q$ unless there is an obvious reason. In this section, we continue to study these graphs and further investigate how they compare to Paley graphs from different perspectives.

\subsection{Maximum cliques come in pairs.}
\label{subsect:even-number-of-cliques} 
 In this subsection, we investigate a pairing between maximum cliques in a semi-primitive pseudo-Paley graph on $q=p^{4t}$ vertices. The goal is to prove
Theorem~\ref{theorem:even-number-cliques} on the parity of the number of canonical cliques: we will see that they come in pairs, and the two cliques in each pair are interchanged by the Galois automorphism $x\mapsto x^{\sqrt{q}}$. 

\begin{lem}\label{lem:C_0-contains-Fpt}
If $p^t \equiv -1 \pmod {2d}$, and $q=p^{2rt}$, then
$\F_{p^t}^* \subset C_0=\langle g^{2d} \rangle$.
\end{lem}

\begin{proof}
Note that $g^k$ is a primitive root of $\F_{p^t}$, where $$
k=\frac{p^{2rt}-1}{p^t-1}=\sum_{j=0}^{2r-1} p^{jt} \equiv \sum_{j=0}^{2r-1} (-1)^j \equiv 0 \pmod {2d}
$$
since $p^t \equiv -1 \pmod {2d}$. Thus, $g^k \in \langle g^{2d} \rangle$, and it follows that $\F_{p^t}^* \subset C_0$.
\end{proof}

The following lemma is a simple consequence of Theorem~\ref{theorem:main-result} and Lemma~\ref{lem:Paley-equal-contribution}. 

\begin{lem}\label{lem:subfield-clique}
Let $X=PP(q, 2d, I)$ be a semi-primitive pseudo-Paley graph with $q=p^{2rt}$ where $r$ is even, and $I \neq \{0, 2, \ldots, 2d-2\}$. Then, the subfield $\F_{\sqrt{q}}$ is not a clique in $X$.
\end{lem}

\begin{proof}
If $\F_{\sqrt{q}}$ were a clique in $X$, then it would be a maximum clique. Comparing equations \eqref{main} and \eqref{subfield}, we get $I=\{0, 2, \ldots , 2d-2\}$, a contradiction. \end{proof}

\begin{lem}\label{lem:group-theory}
Let $g$ be a primitive root of $\F_{q}$ where $q=p^{4t}$ and $p^t\equiv -1\pmod{2d}$. Set $C_0=\langle g^{2d}\rangle$, and pick $x\in C_0$. If $\{x, x^{p^{2t}}\}$ is linearly dependent over $\F_{p^t}$, then $x\in \langle g^{(p^{2t}+1)d} \rangle$. 
\end{lem}

\begin{proof}
If $\{x, x^{p^{2t}}\}$ is linearly dependent over $\F_{p^t}$, then $x^{p^{2t}-1}\in \F_{p^t}^{\ast}$. After expressing $x=g^k$ for some integer $k$, we obtain $g^{(p^{2t}-1)k}\in \F_{p^t}^{\ast}$. As $\F_{p^t}^{\ast}=\langle g^{(q-1)/(p^t-1)}\rangle$, and $q=p^{4t}$ we must have
$$
\frac{q-1}{p^t-1} \mid (p^{2t}-1)k \ \Rightarrow \ (p^{2t}+1)(p^t+1) \mid (p^{2t}-1)k \ \Rightarrow \ p^{2t}+1 \mid (p^t-1)k.
$$
Using $\gcd(p^{2t}+1, p^t-1)=2$, we obtain $\frac{p^{2t}+1}{2} \mid k$. Next, $x=g^k\in C_0=\langle g^{2d}\rangle$  implies that $2d\mid k$. We will now combine these two divisibility relations $\frac{p^{2t}+1}{2} \mid k$ and $2d\mid k$.

Observe that $p^t\equiv -1\pmod{2d}$ implies $p^{2t}\equiv 1\pmod {4d}$, so that $\frac{p^{2t}+1}{2}\equiv 1\pmod{2d}$. Therefore, $\gcd(\frac{p^{2t}+1}{2}, 2d)=1$. Combining $\frac{p^{2t}+1}{2}\mid k$ and $2d\mid k$, we conclude $(p^{2t}+1)d\mid k$ and hence $x=g^k \in \langle g^{(p^{2t}+1)d} \rangle$. 
\end{proof}

We are ready to present the proof of Theorem~\ref{theorem:even-number-cliques}.

\begin{proof}[Proof of Theorem~\ref{theorem:even-number-cliques}]
Without loss of generality, we may assume $m_1=0$. By assumption, $X$ is a Cayley graph on $q=p^{4t}$ many vertices where $p^{t}\equiv -1\pmod{2d}$. Note that the map $f\colon\F_q\to \F_q$ given by $f(x)=x^{\sqrt{q}}$ preserves the edges of the graph: if $x,y \in \F_q$ and $x-y\in D$, then $$f(x)-f(y)=x^{\sqrt{q}}-y^{\sqrt{q}} = (x-y)^{\sqrt{q}} \in D$$ because $\sqrt{q}=p^{2t}\equiv 1 \pmod{2d}$. Note that $f$ fixes each element of $\F_{\sqrt{q}}$ pointwise; in particular, $f$ fixes $\F_{p^t}$. It is also evident that $f$ is a bijection on the vertex set $\F_q$. We conclude that $f$ is a graph automorphism of $X$ which fixes $\F_{p^t}$, and hence must send maximum cliques containing $\F_{p^t}$ to maximum cliques containing $\F_{p^t}$. Let $A$ be a clique of size $\sqrt{q}$, such that $A$ is a $\F_{p^t}$-subspace containing $\F_{p^t}$. We claim that $f(A)\neq A$. The theorem immediately follows from this claim since $f$ is an involution.

We assume, to the contrary, that $f(A)=A$. Since $I$ is different from $\{0, 2, \ldots, 2d-2\}$, the subfield $\F_{\sqrt{q}}$ is not a clique in $X$ by Lemma~\ref{lem:subfield-clique}. Since $A\subset \F_{q}=\F_{p^{4t}}$ is a $2$-dimensional subspace over $\F_{p^{t}}$ containing $\F_{p^t}$ and $A\neq \F_{\sqrt{q}}=\F_{p^{2t}}$, we must have $A\cap \F_{\sqrt{q}}=\F_{p^{t}}$. We will estimate the number of $x\in C_0\cap A$ such that $x$ and $f(x)=x^{p^{2t}}$ are linearly independent over $\F_{p^{t}}$, and use Theorem~\ref{theorem:main-result} to show that there are too many such elements when $f(A)=A$. 

Let $V$ be the set of elements $x$ in $C_0 \cap A$ such that $\{x,x^{p^{2t}}\}$ are dependent over $\F_{p^t}$. By Lemma~\ref{lem:group-theory}, we have $V \subset \langle g^{(p^{2t}+1)d}\rangle$. As $g^{(p^{2t}+1)d}\in \langle g^{\sqrt{q}+1}\rangle = \F_{\sqrt{q}}$, we have $V\subset \F_{\sqrt{q}}$.  In particular, $V \subset A\cap \F_{\sqrt{q}}=\F_{p^t}$. On the other hand, it is clear that $\F_{p^{t}}^{\ast}\subset V$ and $0\notin V$. Therefore, $V=\F_{p^t}^{\ast}$. Theorem~\ref{theorem:main-result} states that $|C_0 \cap A|=\frac{p^{2t}-1}{d}$, which implies that there are $\frac{p^{2t}-1}{d}-(p^t-1)$ elements $x$ in $C_0 \cap A$ such that $\{x,x^{p^{2t}}\}$ are independent over $\F_{p^t}$, that is, the complement $V':=C_0 \cap (A\setminus \F_{p^t}^{\ast})$ has  $\frac{p^{2t}-1}{d}-(p^t-1)$  elements.

Since $f$ fixes $A$, and $A$ is a $\F_{p^t}$-subspace, we can apply Lemma~\ref{lem:C_0-contains-Fpt} to obtain the following inclusion of sets:
$$
\bigcup_{x \in V'}  \big(x\F_{p^t}^{\ast} \cup  x^{p^{2t}}\F_{p^t}^{\ast} \big) \subset C_0 \cap (A\setminus \F_{p^t}^{\ast})=V'.
$$
Note that $|x\F_{p^t}^* \cup  x^{p^{2t}}\F_{p^t}^{\ast}|=2(p^t-1)$ for each $x \in V'$. If $x, y \in V'$, such that $x\F_{p^t}^{\ast} \cup  x^{p^{2t}}\F_{p^t}^{\ast}=y\F_{p^t}^{\ast} \cup  y^{p^{2t}}\F_{p^t}^{\ast}$, then $y \in x\F_{p^t}^{\ast} \cup x^{p^{2t}}\F_{p^t}^{\ast}$. Therefore,
$$
|V'| \geq \bigg |\bigcup_{x \in V'}  \big(x\F_{p^t}^{\ast} \cup  x^{p^{2t}}\F_{p^t}^{\ast} \big) \bigg| \geq 2(p^t-1)\bigg\lceil \frac{|V'|}{2(p^t-1)} \bigg\rceil. 
$$
This implies that $\frac{|V'|}{2(p^t-1)}$ is an integer. However, $p^t \equiv -1 \pmod {2d}$ implies that $\frac{|V'|}{2(p^t-1)}=\frac{p^t+1}{2d}-\frac{1}{2} \notin \Z$, a contradiction. We conclude that $f(A)\neq A$, which completes the proof.
\end{proof}

\begin{ex}
If $X$ is the Peisert graph with order $q=p^4$ where $p>3$, then there are $0$ maximum cliques of size $\sqrt{q}$ with the subspace structure because $\F_p$ forms a maximal clique in $P_q^*$ \cite{AY}*{Theorem 1.5}.
\end{ex}

\begin{ex}[Graphs of order $625$ and $2401$]
If $X=PP(5^4, 6, I)$ with $I\neq \{0, 2, 4\}, \{1, 3, 5\}$ then $X$ contains either $0$ or $2$ maximum cliques of the kind described in Theorem~\ref{theorem:even-number-cliques}. A similar conclusion holds for $X=PP(7^4, 8, I)$. 
\end{ex}

With the help of SageMath, we noticed that in many examples of semi-primitive pseudo-Paley graphs, there are precisely $2$ cliques of size $\sqrt{q}$ containing $\F_{p^t}$. These examples are presented in Appendix~\ref{sect:numerical-evidence}. This leads us to the following conjecture, which simultaneously strengthens and generalizes Theorem~\ref{theorem:even-number-cliques}.

\begin{conj}\label{conj:2-max-cliques}
Let $X=PP(q, 2d, I)$ be a semi-primitive pseudo-Paley graph with $q=p^{2rt}$ where $r$ is even. Assume that $0 \in I$ and $I \neq \{0,2,\ldots, 2d-2\}$. If $\omega(X)=\sqrt{q}$, then the number of maximum cliques containing $\F_{p^t}$ is $2$.
\end{conj}

\subsection{The density result.}\label{subsect:density-zero} 
In this subsection, we discuss the density of the graphs achieving the trivial upper bound on the clique number.

Fix a positive integer $N$. Define $\mathcal{PP}(N)$ to be the collection of all semi-primitive pseudo-Paley graphs of the form $X=PP(q, 2d, I)$ with $q=p^{2rt}\leq N$ and $r$ even. Let $\mathcal{F}(N)$ be the subset of $\mathcal{PP}(N)$ consisting of those graphs $X=PP(q, 2d, I)$ with $\omega(X)=\sqrt{q}$. Since $\omega(X)\leq\sqrt{q}$ always holds, the graphs in $\mathcal{F}(N)$ are precisely those that attain the trivial upper bound on the clique number.

\begin{prop}\label{prop:density-zero} Assume that Conjecture~\ref{mainconj} is true. The density of graphs attaining the trivial upper bound on the clique number inside the family of semi-primitive pseudo-Paley graphs is zero; that is, 
$$
\lim_{N \to \infty} \frac{\# \mathcal{F}(N)}{\#\mathcal{PP}(N)}=0.
$$
\end{prop}

\begin{proof}
Given a prime power $q=p^{2rt}\leq N$, with $r$ even, let $\mathcal{PP}(N, p, t) \subset \mathcal{PP}(N)$ be the subset consisting of graphs $PP(q,2d,I)$ with $q=p^{2rt}$ vertices such that $t$ is the smallest positive integer such that $p^t \equiv -1\pmod{2d}$. Similarly, define $\mathcal{F}(N, p, t)=\mathcal{F}(N)\cap \mathcal{PP}(N, p, t)$.  We have the following decomposition depending on $p$ and $t$:
\begin{align*}
\mathcal{PP}(N) = \bigsqcup_{p^t \leq N^{1/4}} \mathcal{PP}(N, p, t), \quad \text{and} \quad \mathcal{F}(N) = \bigsqcup_{p^t \leq N^{1/4}} \mathcal{F}(N, p, t).
\end{align*} 
This is a disjoint union because for any $PP(q, 2d, I) \in \mathcal{PP}(N)$, there is a unique $(p, t)$ such that  $PP(q, 2d, I)\in \mathcal{PP}(N, p, t)$. Indeed, $q$ determines the value of $p$, and the congruence $p^{t}\equiv -1\pmod{2d}$ uniquely determines the value of smallest such $t$. 

Recall that the Grassmannian $\operatorname{Gr}(r, 2r)(\F_{p^t})$ is the set of all $r$-dimensional $\F_{p^t}$-subspaces inside $\F_{(p^t)^{2r}}=\F_{q}$. Given a graph $X=PP(q, 2d, I)\in\mathcal{F}(N, p, t)$ with $q=p^{2rt}$, Conjecture~\ref{mainconj} implies that the set of maximum cliques in $X$ containing $0$ gives a subset $S_{X}\subset \operatorname{Gr}(r, 2r)(\F_{p^t})$ with $S_X \neq \emptyset$. If $Y \in\mathcal{F}(N, p, t)$ with $Y=PP(q,2d,J)$, and $I\neq J$, then we claim that $S_X\cap S_{Y}=\emptyset$. Indeed, by Theorem~\ref{theorem:main-result}, a maximum clique $V \in S_X$ uniquely determines $I$ because $I = \{ j \ : \ 0\leq j\leq 2d-1 \text{ and } |C_j\cap V|=\frac{\sqrt{q}-1}{d}\}$, and hence the graph $X$. If $V\in S_X\cap S_Y$, then $X=Y$. The claim is proved, and we obtain:
$$
\# \mathcal{F}(N, p, t) \leq \# \bigsqcup_{p^{2rt}\leq N} \operatorname{Gr}(r, 2r)(\F_{p^t}).
$$
The cardinality of the Grassmannian can be bounded as follows:
$$
\# \operatorname{Gr}(r, 2r)(\F_{p^t}) = \frac{(p^{2rt}-1)(p^{2rt}-p^t)\cdots (p^{2rt}-p^{(r-1)t})}{(p^{rt}-1)(p^{rt}-p^t)\cdots (p^{rt}-p^{(r-1)t})} 
\leq p^{2r^2t}.
$$
Thus, 
\begin{equation*}
\# \mathcal{F}(N,p,t) \leq \sum_{p^{2rt} \leq N}(p^t)^{2r^2} \leq \ln(N)\cdot N^{r} \leq \ln(N)\cdot N^{\ln(N)} = \ln(N) \cdot e^{(\ln(N))^2},    
\end{equation*}
which implies, 
\begin{equation}\label{FN}
\# \mathcal{F}(N)=\sum_{p^{4t} \leq N} \# \mathcal{F}(N,p,t) \leq N\cdot \ln(N) \cdot e^{(\ln(N))^2}. 
\end{equation}

On the other hand, the number of graphs of the form $PP(q, 2d, I) \in \mathcal{PP}(N, p, t)$ with $q=p^{2rt}$ and $d=(p^t+1)/2$ is given by the binomial coefficient $\binom{2d}{d}=\binom{p^t+1}{(p^t+1)/2}$ since $I\subset\{0, 1, \ldots, 2d-1\}$ with $\# I = d$. Using Stirling's approximation,
\begin{equation*}
\# \mathcal{PP}(N, p, t) \geq  \binom{p^t+1}{(p^t+1)/2} \geq \frac{4^{p^t}}{\sqrt{4p^t}}.
\end{equation*}

By Chebyshev's Theorem, there exists a prime $p_0$ such that $\frac{1}{2}N^{1/4}<p_0<N^{1/4}$ for $N$ large. As a result, 
\begin{equation}\label{PPN}
\#\mathcal{PP}(N)\geq \#\mathcal{PP}(N, p_0, 1) \geq \frac{4^{p_0}}{\sqrt{4p_0}}\geq \frac{2^{N^{1/4}}}{\sqrt{2N^{1/4}}}.
\end{equation}
Combining inequalities \eqref{FN} and \eqref{PPN}, 
$$
\limsup_{N\to\infty} \frac{\#\mathcal{F}(N)}{\#\mathcal{PP}(N)} \leq \limsup_{N\to\infty} \frac{N\cdot \ln(N) \cdot e^{(\ln(N))^2}}{ \frac{2^{N^{1/4}}}{\sqrt{2N^{1/4}}}} = 0
$$
as desired. \end{proof}

We remark that Proposition~\ref{prop:density-zero} can be phrased unconditionally. The argument shows that a union of at most $d$ distinct $2d$-th semi-primitive cyclotomic classes of $\F_q$ together with $\{0\}$, almost surely does not contain a subspace of size $\sqrt{q}$, as $q \to \infty$. 

\subsection{An infinite family of graphs (of non-Paley type) achieving the square root upper bound.} \label{subsect:heuristic-kernel} 

Despite the density of $\mathcal{F}(N)$ being zero, we still expect that $\mathcal{F}(N)$ will contain infinitely many elements that are non-isomorphic to the Paley graph. Specializing to the case $q=p^4$ and $2d=p+1$, the proof of Proposition~\ref{prop:density-zero} can be modified to show that $\omega(PP(p^4, p+1, I))<p^2$ for a generic choice of $I$. Nonetheless, we expect that there are infinitely many exceptions of non-Paley type. This expectation is made precise in the following conjecture. 

\begin{conj}\label{conj:josh} For each prime $p\geq 3$, 
 there are exactly $\frac{p^2+3}{4}$ subsets $I\subset \{0, 1, \ldots, p\}$ with $0\in I$ and $|I|=\frac{p+1}{2}$ such that $\omega(PP(p^4, p+1, I))=p^2$. Furthermore, in these graphs, the maximum cliques containing $\{0,1\}$ are of the form $\F_p \oplus a\F_p$, where $a=g^{(p+1)k}$ and $k$ is an odd integer.
\end{conj}

Conjecture~\ref{conj:josh} is impractical to verify using a computer. Instead, we consider the following weaker conjecture, which can be verified in polynomial time via Algorithm~\ref{alg:josh-question} in Appendix~\ref{sect:numerical-evidence}.

\begin{conj}\label{conj:subspace}
Let $V$ be a $2$-dimensional subspace in $\F_{p^4}$, such that $1 \in V$. Then $V$ is a clique in $PP(p^4,p+1,I)$ for some $I$ if and only if $V=\F_p \oplus a\F_p$, where $a=g^{(p+1)k}$ and $k$ is an odd integer.
\end{conj}

The following proposition states that Conjecture~\ref{conj:subspace} is essentially equivalent to Conjecture~\ref{conj:josh} if we assume the main Conjecture~\ref{mainconj}.

\begin{prop}
Conjecture~\ref{mainconj}, Conjecture~\ref{conj:2-max-cliques}, and 
Conjecture~\ref{conj:subspace} together imply Conjecture~\ref{conj:josh}.
\end{prop}

\begin{proof}
Suppose $I\subset \{0, 1, \ldots, p\}$ with $0\in I$ and $|I|=\frac{p+1}{2}$ such that $\omega(PP(p^4, p+1, I))=p^2$. Let $V$ be a maximum clique of size $p^2$ with $0, 1\in V$; see the beginning of the proof of Theorem~\ref{thm:clique_number} for the existence of such a clique.
Conjecture~\ref{mainconj} implies that $V$ is a $2$-dimensional $\F_p$-subspace. We apply Conjecture~\ref{conj:subspace} to express $V=\F_p\oplus a\F_p$ for some $a=g^{(p+1)k} \in \F_{p^4}$, where $k$ is an odd integer. 

Recall that $C_0=\langle g^{p+1}\rangle$ is a $(p+1)$-th cyclotomic class. Lemma~\ref{lem:C_0-contains-Fpt} states that $\F_p^* \subset C_0$. Observe that $V \cap C_0 \supset \F_p^* \cup a\F_p^*$. By Theorem~\ref{theorem:main-result}, $|V\cap C_0|=\frac{p^2-1}{(p+1)/2}=2(p-1)$. Hence, $V \cap C_0 = \F_p^* \cup a\F_p^*$. Let $V'=\F_p \oplus b\F_p$ be another clique of the same type. The same argument shows that $V'\cap C_0=\F_p^* \cup b\F_p^*$. Thus, we have $V=V'$ if and only if $b/a \in \F_p^*$. 

Given $V=\F_p^* \oplus a\F_p^*$ with $a=g^{(p+1)k}$ with $k$ odd, the number of possible $k$'s is $\frac{p^4-1}{2(p+1)}=(p^2+1)(p-1)/2.$ The argument above shows that different choices of $a$ in the same $\F_{p}^{\ast}$-coset give the same space $V$. Therefore, the number of choices for $V$ is:
$$
\frac{1}{2} \frac{(p^2+1)(p-1)}{p-1}=\frac{p^2+1}{2}.
$$

If $I=\{0, 2, 4, \ldots, p-1\}$ then $PP(p^4, p+1, I)$ is the Paley graph and the only possible $V$ in this case is $V=\F_{p^2}$. If $I\neq \{0, 2, 4, \ldots, p-1\}$, then Conjecture~\ref{conj:2-max-cliques} yields $2$ such subspaces $V$. By Theorem~\ref{theorem:main-result}, different choices of $I$ give rise to different subspaces $V$. In summary, there are exactly
$$
\frac{1}{2} \bigg(\frac{p^2+1}{2} - 1\bigg) + 1  = \frac{p^2+3}{4}
$$
many $I$ with $0\in I$ satisfying $\omega(PP(p^4, p+1, I))=p^2$. \end{proof}

\subsection{Conditional improvement on the trivial upper bound}\label{subsect:conditional-improvement}
In this subsection, we prove Theorem~\ref{thm:clique_number} and Theorem~\ref{thm: Peisert_clique_number}, which improve the previously known upper bound $\sqrt{q}$ on the clique number conditional on Conjecture~\ref{mainconj}.

The following lemma is useful for obtaining a lower bound on the character sum, as we will see in the proof of Theorem~\ref{thm:clique_number} below.

\begin{lem}\label{lem:theta-sum-lb}
Let $d\geq 2$ be a positive integer and $\theta$ be a primitive $2d$-th root of unity. Given a subset $I=\{m_1, m_2, \ldots,m_d\}\subset \{0, 1, \ldots, 2d-1\}$, there exists an integer $1\leq k\leq 2d-1$ such that $|\sum_{j=1}^{d} \theta^{k m_j}|>\sqrt{d/2}$. 
\end{lem}
\begin{proof}
We have
$$
\sum_{k=0}^{2d-1}\left|\sum_{j=1}^{d} \theta^{km_j}\right|^2=\sum_{k=0}^{2d-1} \sum_{1\leq j,\ell \leq d} \theta^{k(m_j-m_{\ell})}=\sum_{1\leq j,\ell \leq d} \sum_{k=0}^{2d-1} \theta^{k(m_j-m_{\ell})}=2d^2.
$$
It follows that
\[\max_{1 \leq k \leq 2d-1}\left|\sum_{j=1}^{d} \theta^{km_j}\right| \geq \sqrt{\frac{2d^2-d^2}{2d-1}}=\sqrt{\frac{d^2}{2d-1}}>\sqrt{\frac{d}{2}}. \qedhere
\]
\end{proof}

The following theorem is a consequence of the main result in a recent paper by Reis \cite{Reis}. It states that there will be a lot of cancellations in a character sum over a generic subspace $V$ of $\F_{q}$ provided that $V$ is not the subfield $\F_{\sqrt{q}}$ and $q$ is sufficiently large. 

\begin{thm}[\cite{AY}*{Corollary 3.6}]\label{charsumcor}
Let $n$ be an integer such that $n \geq 2$, and $q$ an odd prime power. Let $V \subseteq \mathbb{F}_{q^{{2n}}}$ be an $\mathbb{F}_{q}$-subspace of dimension $n$, with $1 \in V$, and $V \neq \mathbb{F}_{q^{n}}$. Then for any non-trivial multiplicative character $\chi$ of $\mathbb{F}_{q^{2n}}$,
\begin{equation} \label{charsum}
\left|\sum_{x \in V} \chi(x)\right|< \frac{2n}{\sqrt{q}} \cdot |V|.    
\end{equation}
\end{thm}

In other words, if the absolute value of the character sum over a subset $S \subset \F_q$ with size $\sqrt{q}$ is too large, then we can conclude that either $S$ is not a subspace or $S$ is the subfield $\F_{\sqrt{q}}$. We remark that the second alternative (the subfield case) was one of the main tools used in \cite{AY} to establish a new proof of Theorem~\ref{EKRPaley} as well as its generalization.
Next, we use the first alternative to prove Theorem~\ref{thm:clique_number}. 

\begin{proof} [Proof of Theorem~\ref{thm:clique_number}]
Assume, to the contrary, that $\omega(X)=\sqrt{q}$ and $p^t > 10.2 r^2 d$. Without loss of generality, we may assume $m_1=0$. Since $X$ is a Cayley graph, we can restrict attention to maximum cliques containing $0$. If $A$ is a maximum clique in $X$ containing $0$, Theorem~\ref{theorem:main-result} implies that $C_0 \cap A \neq \emptyset$ and thus we can further assume that $1 \in A$. Indeed, any $y\in C_0 \cap A$ yields another maximum clique $B=y^{-1}A$ containing both $0$ and $1$.  By Lemma~\ref{lem:subfield-clique},  $A \neq \F_{\sqrt{q}}$.

By Lemma \ref{lem:theta-sum-lb}, we can find $1 \leq k \leq 2d-1$, such that $|\sum_{j=1}^{d} \theta^{km_j}| >\sqrt{d/2}$. Let $\chi$ be the multiplicative character of $\F_q$ such that $\chi(g)=\theta^k$. Then $\chi$ is a nontrivial character with order dividing $2d$.  Suppose that $A$ is a subspace over $\F_{p^t}$; then by Theorem~\ref{charsumcor} and Theorem~\ref{theorem:main-result}, we have
\begin{equation}\label{charsumcondition}
\frac{\sqrt{q}-1}{d} \left|\sum_{j=1}^{d} \theta^{km_j}\right| = \left|\sum_{j=1}^{d} \sum_{x\in C_{m_j}} \chi(x)\right| = \left|\sum_{x \in A} \chi(x)\right| < \frac{2r}{\sqrt{p^t}} \sqrt{q},    
\end{equation}
Thus,
$$
\frac{\sqrt{q}-1}{\sqrt{2d}} \leq \frac{\sqrt{q}-1}{d} \left|\sum_{j=1}^{d} \theta^{km_j}\right| < \frac{2r}{\sqrt{p^t}} \sqrt{q}.
$$
However, the above inequality implies $p^t \leq 10.2 r^2 d$, a contradiction. This finishes the proof.
\end{proof}

While Theorem~\ref{thm:clique_number} is stated conditionally in terms of the clique number of certain Cayley graphs, it can also be phrased unconditionally with additive combinatorics flavor as follows.

\begin{cor}\label{cor:clique-number-additive}
Let $q=p^{2rt}$, where $r$ is even, and $t$ is the smallest positive integer such that $p^t \equiv -1 \pmod {2d}$. Let $D$ be the union of at most $d$ distinct $2d$-th cyclotomic classes of $\F_q$. Assume that $D$ contains at least one square and one non-square element. If $p^t>10.2r^2d$, then there is no $\F_{p^t}$-subspace $V$ of $\F_q$, with dimension at least $r$, such that $V \subset D \cup \{0\}$.
\end{cor}

When the set $I$ has a simple structure, the lower bound on $p$ in Theorem~\ref{thm:clique_number} can be further improved. We illustrate this by analyzing the case of generalized Peisert graphs, where we can explicitly compute the sum of roots of unity.

\begin{cor}\label{cor:generalized-Peisert}
Let $q=p^{2rt}$, where $r$ is even and $t$ is the smallest positive integer such that $p^t \equiv -1 \pmod {2d}$. Assuming Conjecture \ref{mainconj}, the generalized Peisert graph $GP^*(q,2d)$ has clique number less than $\sqrt{q}$ when $p^t\geq 12.5r^2$.
\end{cor}

\begin{proof}
We will assume that $\omega(GP^{\ast}(q, 2d))=\sqrt{q}$ and show that $p^t < 12.5r^2$ must hold. Recall that $GP^{\ast}(q, 2d)=PP(q,2d,I)$, where $I=\{0,1,\ldots, d-1\}$; see Example~\ref{ex:pp}.

Let $\theta=e^{2\pi i/2d}=e^{i \pi/d}$. Let $\chi$ be the multiplicative character of $\F_q$ such that $\chi(g)=\theta$. We have
$$
\sum_{j=1}^{d} \theta^{m_j}=\sum_{j=1}^{d}\theta^{j-1}=\frac{1-\theta^d}{1-\theta}=\frac{2}{1-\theta}.
$$
Using the inequality $\cos x \geq 1-\frac{x^2}{2}$ for all real $x$, we have
 $$|1-\theta|=\sqrt{2-2 \cos \frac{\pi}{d}} \leq \sqrt{2-2\bigg(1-\frac{\pi^2}{2d^2}\bigg)}=\frac{\pi}{d},$$
and so,
\begin{equation}\label{eq:char-sum-generalized-peisert}
\bigg|\frac{2}{1-\theta}\bigg| \geq  \frac{2d}{\pi}.
\end{equation}
Combining the inequalities \eqref{charsumcondition} and \eqref{eq:char-sum-generalized-peisert} yields
$$
\frac{\sqrt{q}-1}{d} \cdot \frac{2d}{\pi} \leq  \frac{\sqrt{q}-1}{d} \left|\sum_{j=1}^{d} \theta^{m_j}\right| <  \frac{2r}{\sqrt{p^t}} \sqrt{q}.
$$
This last inequality rearranges into:
\begin{equation} \label{125r2}
\sqrt{p^t} < \frac{r\pi \sqrt{q}}{\sqrt{q}-1} \ \ \Rightarrow \ \ p^t < \frac{r^2 \pi^2 q}{(\sqrt{q}-1)^2} \leq \left(\frac{81}{64}\right)r^2\pi^2< 12.5r^2
\end{equation}
since $q \geq 81$. \end{proof}

To prove Theorem~\ref{thm: Peisert_clique_number}, we apply the above corollary and then use SageMath~\cite{sagemath} to verify the cases when $q$ is small. 

\begin{proof} [Proof of Theorem~\ref{thm: Peisert_clique_number}] We assume the clique number is $\sqrt{q}$ and apply the contrapositive of Corollary~\ref{cor:generalized-Peisert} in the special case $r=2$ to conclude that $p^t < 50$. Note that when $p^t \geq 41$, $q=p^{4t} \geq 41^4$, and inequality~\eqref{125r2} implies that
$$
p^t < 4\pi^2\bigg(\frac{1681}{1680}\bigg)^2 <40,
$$
a contradiction. Therefore, 
$$p^t \in \{5,7,9,11,13, 17,19,23,25, 27, 29,31,37\}.
$$
Recall that $p^t \equiv -1 \pmod {2d}$ and $d \geq 2$. However, using Algorithm~\ref{alg:clique-number} (described in Appendix~\ref{sect:numerical-evidence}), we verified that the corresponding graph for each choice of $(p^t,d)$ in the following list has clique number less than $\sqrt{q}=p^{2t}$ leading to a contradiction:
\begin{align*}
\{& (5,3), (7,2), (7,4), (9,5), (11,2), (11,3), (11,6), (13,7), (17,3), (17,9),\\
&(19,5), (19,10), (23,2), (23,3), (23,4), (23,6), (23,12), (25, 13), (27, 7),\\
&(27, 14), (29, 3), (29,5), (29,15), (31, 2), (31,4), (31,8), (31,16), (37, 19)\}.     
\end{align*}
The largest case $(p^t, d)=(37, 19)$ took about 16 hours in SageMath. 
\end{proof}

\section*{Acknowledgements}
The authors thank Gabriel Currier, Greg Martin, Sergey Goryainov, Zinovy Reichstein, and J\'ozsef Solymosi for helpful discussions. The authors thank the anonymous referees for their valuable comments and suggestions. During the preparation of this manuscript, the first author was supported by a postdoctoral research fellowship from the University of British Columbia and the NSERC PDF award.

\bibliographystyle{abbrv}
\bibliography{biblio.bib}

\appendix

\section{Gauss sums and character sums}\label{sect:gauss-sums}

In this section, we work with Gauss sums and character sums over finite fields. We refer to \cite{BEW98} and \cite{LN}*{Chapter 5} for a detailed discussion. This appendix contains preparatory material for Appendix~\ref{sect:proof-of-main-prop} on the proof of Proposition~\ref{proposition:main-result}.

Let $\chi$ be a multiplicative character of $\F_q$, and $\psi$ an additive character of $\F_q$. As usual, we extend $\chi$ to be defined on $\F_q$ by setting $\chi(0)=0$. The {\em Gauss sum} associated to $\chi$ and $\psi$ is defined to be $G(\chi, \psi)=\sum_{ c \in \F_q} \chi(c) \psi(c).$
When $\psi$ is the canonical additive character $e_p\big(\Tr_{\F_q}(\cdot)\big)$, we use $G(\chi)$ to denote $G(\chi,\psi)$. The first nontrivial example is when $\chi$ is a quadratic character; in this case, the Gauss sum is well understood according to the following classical theorem.

\begin{thm}[\cite{BEW98}*{Theorem 11.5.4}] \label{quad}
Let $\F_q$ be a finite field with $q=p^s$, where $p$ is an odd prime, and $s$ is a positive integer. Let $\chi$ be the quadratic character of $\F_q$. Then
$$G(\chi)=\begin{cases}
(-1)^{s-1}\sqrt{q},  \quad \quad \text{if } p \equiv 1 \pmod 4,\\
(-1)^{s-1}i^s\sqrt{q},  \quad \text{ if } p \equiv 3 \pmod 4.
\end{cases}$$
\end{thm}

\begin{cor}\label{quadevenr}
Suppose $p$ is a prime such that $p^{t}\equiv -1\pmod{2d}$ with $t$ minimal. Let $q=p^{2rt}$, where $r$ is even. If $\chi$ be the quadratic character of $\F_q$, then $G(\chi)=-\sqrt{q}$. 
\end{cor}
\begin{proof}
We can write $q=p^{s}$ where $s=2rt$. Note that $i^{2rt}=1$ since $r$ is even. The corollary follows from Theorem \ref{quad}.
\end{proof}

Stickelberger's theorem (see, for example, \cite{BEW98}*{Theorem 11.6.3}) provides an explicit formula to compute semi-primitive Gauss sums. The original proof by Stickelberger \cite{Sti90} was based on algebraic geometry. 

\begin{thm}[Stickelberger's theorem] \label{forr} 
Let $p$ be an odd prime and $d>2$ be an integer. Suppose that there exists a
positive integer $t$ such that $p^t\equiv -1\pmod{d}$, with $t$ chosen to be minimal. Let $\chi$ be a multiplicative character of $\mathbb{F}_{p^v}$ with order $d$.  Then $v=2ts$ for some positive integer $s$, and
$
p^{-v/2}G(\chi)=(-1)^{s-1+(p^{t}+1)s/d}.
$
\end{thm}

According to \cite{BWX}*{Proposition 1}, the following corollary is a consequence of Stickelberger's theorem and Davenport-Hasse theorem: an equivalent statement in terms of the generating function of Gauss periods first appeared (implicitly) in \cite{BM72}*{Section 5}. For the sake of completeness, we present a self-contained proof below.

\begin{cor}\label{Gausssum}
Let $q=p^{2rt}$, where $p^{t} \equiv -1 \pmod{2d}$ is a prime with $t$ minimal. Let $\chi$ be a multiplicative character of $\F_q$ with order $2d$. Then for each $1 \leq k \leq 2d-1$, $G(\chi^k)=-\sqrt{q}$.
\end{cor}

\begin{proof}
When $k=d$, then $\chi^{d}$ is the quadratic character and  $G(\chi^d)=-\sqrt{q}$ by Corollary~\ref{quadevenr}.
 
Next, we assume that $d \nmid k$. Then $\chi^k$ is a character with order $e=2d/\gcd(2d,k)>2$, where $e \mid 2d$. Since $-1$ is a power of $p \pmod {2d}$, $-1$ is also a power of $p \pmod {e}$. Let $t'$ be the smallest positive integer such that $p^{t'} \equiv -1 \pmod e$; then it is easy to check that the order of $p \pmod e$ is $2t'$. Note that $p^t \equiv -1 \pmod {2d}$; in particular, $p^t \equiv -1 \pmod {e}$ and $p^{t-t'} \equiv 1 \pmod {e}$. Thus, $2t' \mid (t-t')$ so that $t/t'$ is an odd integer. Let $r'$ be defined by the equation $rt=r't'$ so that $q=p^{2rt}=p^{2r't'}$. It is clear that $r$ and $r'$ have the same parity. 

We claim that $p^{t}\equiv p^{t'}\pmod{2e}$. We consider two cases. 
\begin{enumerate}[(i)]
\item Suppose that $p^{t'} \equiv -1 \pmod {2e}$. Since $t/t'$ is odd, we also have $p^{t} \equiv -1 \pmod{2e}$. 
\item Suppose that $p^{t'}\equiv e-1 \pmod{2e}$. 
\begin{enumerate} 
\item When $e$ is even, observe that $(e-1)^2\equiv 1\pmod{2e}$, and so $(e-1)^{t/t'} \equiv e-1 \pmod{2e}$.
\item  When $e$ is odd, $(e-1)^2\equiv e^2+1\equiv e+1\pmod{2e}$; now, $(e+1)^2\equiv e^2+1\equiv e+1 \pmod{2e}$ and so $(e-1)^{t/t'} \equiv (e+1)(e-1)\equiv e^2-1\equiv e-1 \pmod{2e}$. 
\end{enumerate} In both subcases, $(e-1)^{t/t'}\equiv e-1\pmod{2e}$. As a result, $p^{t}\equiv e-1\pmod{2e}$.
\end{enumerate} 
This proves the claim. Consequently, 

\begin{enumerate}[(i)]
\item If $p^t \equiv -1 \pmod {4d}$, then $p^{t'} \equiv p^t \equiv -1 \pmod {2e}$.
\item If $p^t \equiv 2d-1 \pmod {4d}$, then we consider two subcases.
\begin{enumerate} 
\item if $e \mid d$ (i.e., $k$ is even), then $p^{t'} \equiv p^t \equiv -1 \pmod {2e}$.
\item if $e \nmid d$ (i.e., $k$ is odd), then $p^{t'}\equiv p^t \equiv e-1 \pmod {2e}$.
\end{enumerate}
\end{enumerate} 
Finally, we apply Theorem~\ref{forr} to obtain:
\begin{align*}
     G(\chi^k)/\sqrt{q}= (-1)^{r'-1+(p^{t'}+1)r'/e} =
    (-1)^{r'-1}=(-1)^{r-1}=-1.
\end{align*}
as desired. 
\end{proof}

The following theorem, which can be proved directly using orthogonality relations, expresses a complete character sum with a monomial argument as a linear combination of Gauss sums. 

\begin{thm}[\cite{LN}*{Theorem 5.30}] \label{Weilsum}
Let $\psi$ be a nontrivial additive character of $\F_q$, $n \in \N$, and $\chi$ a multiplicative character of $\F_q$ of order $e=\gcd(n,q-1)$. Then for any $a \in \F_q^*$,
$$
    \sum_{c \in \F_q} \psi(ac^n)=\sum_{j=1}^{e-1} \bar{\chi}^j(a) G(\chi^j, \psi).
$$
\end{thm}

Theorem~\ref{Weilsum} allows us to explicitly compute the following {\em Gauss periods}, that is, the character values of cyclotomic classes.

\begin{cor}\label{subsum}
Let $q=p^{2rt}$, where $p^{t} \equiv -1 \pmod{2d}$ is a prime with $t$ minimal and $r$ is even. Let $C_0, C_1, \ldots, C_{2d-1}$ be the list of $2d$-th cyclotomic classes of $\F_q$. Then
$$
\sum_{c \in C_0} e_p\big(\Tr_{\F_q}( c)\big)=-\frac{(2d-1)\sqrt{q}+1}{2d}, \quad \text{and} \quad \sum_{c \in C_k} e_p\big(\Tr_{\F_q}(c)\big)=\frac{\sqrt{q}-1}{2d}
$$
for each $1\leq k\leq 2d-1$. 

\end{cor}

\begin{proof}
Let $\psi$ denote the canonical additive character of $\F_{q}$ and let $\chi$ be a multiplicative character such that $\chi(g)=\theta$, where $\theta$ is a primitive $2d$-th root of unity. Then $\chi$ has order $2d$, and Corollary \ref{Gausssum} implies that $G(\chi^j)=-\sqrt{q}$ for each $1 \leq j \leq 2d-1$. 

Using Theorem~\ref{Weilsum} with $n=e=2d$, for any $a \in \F_q^*$, we have
\begin{align}\label{eq:weil-sum-monomial}
\sum_{c\in\F_q} e_p(\Tr_{\F_q}(ac^{2d}))=\sum_{j=1}^{2d-1} \bar{\chi}^j(a) G(\chi^j)=-\sqrt{q} \sum_{j=1}^{2d-1} \bar{\chi}^j(a)=\sqrt{q}-\sqrt{q} \sum_{j=0}^{2d-1} \bar{\chi}^j(a).
\end{align}
By the orthogonality relations, 
\begin{align}\label{eq:orthog-rel}
\sum_{j=0}^{2d-1}\bar{\chi}^j(a) = \begin{cases}
1 &\text{if } a=1 \\
0 &\text{if } a\neq 1
\end{cases}
\end{align}
Combining  \eqref{eq:weil-sum-monomial} and \eqref{eq:orthog-rel}, we obtain
$$
\sum_{c\in\F_q} e_p(\Tr_{\F_q}(c^{2d}))=-(2d-1)\sqrt{q}, \quad \text{and} \quad  \sum_{c\in\F_q} e_p(\Tr_{\F_q}(g^kc^{2d}))=\sqrt{q}
$$
for each $1 \leq k \leq 2d-1$. The desired conclusion follows immediately from $e_p(\Tr_{\F_q}(0))=1$.
\end{proof}

The following lemma allows us to compute character values using the associated Gauss sum.

\begin{lem} [\cite{LN}*{Theorem 5.12}] \label{qf}
Let $\chi$ be a multiplicative character of  $\F_q$. Then for any $a \in \F_q$, 
$$
\overline{\chi(a)}=\frac{1}{G(\chi)} \sum_{ c \in \F_q} \chi(c) e_p\big(\Tr_{\F_q}(ac)\big). 
$$
\end{lem}

The lemma below, which is an application of Lemma~\ref{qf} to double character sums, appeared implicitly in \cite{YipG}*{Proposition 4.2}.

\begin{lem}\label{complete}
Let $q$ be an odd prime power. Let $A$ be a subset of $\F_q$, and let $\chi$ be a nontrivial multiplicative character of $\F_q$. Then
$$
\sum_{a,b\in A}\overline{\chi(a-b)}=\frac{1}{G(\chi)} \sum_{c \in \F_q^*} \chi(c) |S(q,A;c)|^2.
$$
\end{lem}

\begin{proof}
Recall that the trace map is linear. By Lemma \ref{qf}, 
\begin{align*}
\sum_{a,b\in A}\overline{\chi(a-b)}
&=\sum_{a,b\in A} \frac{1}{G(\chi)} \sum_{c \in \F_q^*} \chi(c)  e_p\big(\Tr_{\F_q}((a-b)c)\big)\\
&=\frac{1}{G(\chi)} \sum_{c \in \F_q^*} \chi(c) \sum_{a,b\in A} e_p\big(\Tr_{\F_q}((a-b)c)\big)\\
&=\frac{1}{G(\chi)} \sum_{c \in \F_q^*} \chi(c) \bigg(\sum_{a\in A} e_p\big(\Tr_{\F_q}(ac)\big)\bigg) \bigg(\sum_{b\in A} e_p\big(\Tr_{\F_q}(-bc)\big)\bigg)    \\
&=\frac{1}{G(\chi)} \sum_{c \in \F_q^*} \chi(c) \bigg(\sum_{a\in A} e_p\big(\Tr_{\F_q}(ac)\big)\bigg) \bigg(\sum_{b\in A} e_p\big({-}\Tr_{\F_q}(bc)\big)\bigg)    \\
&=\frac{1}{G(\chi)} \sum_{c \in \F_q^*} \chi(c) |S(q,A; c)|^2. \qedhere
\end{align*}
\end{proof}

Recall that $S(q,A;c)$, defined in equation~\eqref{Sqac}, is a scaled Fourier coefficient of the indicator function on $A$. Thus, Plancherel's identity implies the following lemma.

\begin{lem}[\cite{YipG}*{Lemma 4.1}]\label{sumc}
For any $A \subset \F_q$, we have
$$
\sum_{c \in \F_q^*} \big|S(q,A;c)\big|^2=q|A|-|A|^2.
$$
\end{lem}

\section{Proof of Proposition~\ref{proposition:main-result}}\label{sect:proof-of-main-prop}
Let $A$ be a maximum clique in $X$. We aim to show the following two statements:
\begin{enumerate}
    \item $|A|\leq\sqrt{q}$; \label{item:trivial-upper-bound}
    \item $|A|=\sqrt{q}$ if and only if $S(q,A;c)=0$ for all $c \in D'$. \label{item:characterization}
\end{enumerate} 
We have mentioned earlier that \eqref{item:trivial-upper-bound} can be proved independently using the Delsarte bound. The proof given here is genuinely different, and we include it as \eqref{item:trivial-upper-bound} and \eqref{item:characterization} naturally fit together. In the following, we use $S(c)$ to denote $S(q, A; c)$ for each $c\in\F_q^{\ast}$, and use $T_k=\sum_{c\in C_k} |S(c)|^2$ as a shorthand for each $0\leq k\leq 2d-1$.

Let $\theta$ be a primitive $2d$-th root of unity. Let $\chi$ be the multiplicative character of $\F_q$ such that $\chi(g)=\theta$. In order to study the structure of $A$, we associate to the graph $X$ the following interpolation function $f\colon\F_q \to \R$ defined by
$$
f(x)=\sum_{j=0}^{2d-1} c_j \overline{\chi^j(x)} \quad \quad \text{where } c_j=\sum_{\ell=1}^{d} \theta^{jm_\ell}.
$$
We claim that $f/2d$ is the indicator function on the connection set $D$ of the Cayley graph $X$:
\begin{equation}\label{sumpp}
f(x)=
\begin{cases}
2d, & \text{if } x \in D,\\
0,  & \text{if } x \notin D,
\end{cases}    
\end{equation}
where $D$ is defined in equation~\eqref{D}. To prove this claim, recall the orthogonality relations for the group of order $|G|=2d$ generated by the character $\chi$:
$$
\frac{1}{2d}\sum_{j=0}^{2d-1} \chi^j(g^k) \overline{\chi^j(g^{\ell})} = 
\begin{cases} 
1 &\text{if } k= \ell, \\
0 &\text{if } k\neq \ell.
\end{cases}
$$
Next, 
\begin{align*} 
f(x)=\sum_{j=0}^{2d-1} \bigg(\sum_{\ell=1}^{d} \theta^{jm_\ell}\bigg) \overline{\chi^j(x)}=\sum_{j=0}^{2d-1} \bigg(\sum_{\ell=1}^{d} \chi^j(g^{m_{\ell}})\bigg) \overline{\chi^j(x)}=\sum_{\ell=1}^{d} \bigg( \sum_{j=0}^{2d-1}  \chi^j(g^{m_{\ell}}) \overline{\chi^j(x)} \bigg).
\end{align*}
Given $x\in D$, we have $x=g^{2dk+m_{\ell}}$ for some $k\geq 0$ and some $1\leq \ell\leq d$; in this case, there will be exactly one non-zero inner sum, and so $f(x)=f(g^{m_{\ell}})=2d$. If $x=0$, then $f(x)=0$. Given $x\notin D$ such that $x \neq 0$, we have $x=g^{2dk+m}$ for some $k\geq 0$ and $m\neq m_{\ell}$ for any $\ell$; in this case, all the inner sums will vanish and we get $f(x)=f(g^{m})=0$, establishing the claim above. 

Since $A$ is a clique, for any $a,b \in A$, if $a \neq b$, we have $a-b \in D$; otherwise $f(a-b)=0$. Thus, equation \eqref{sumpp} implies that
\begin{equation}\label{clique}
 \sum_{a,b \in A} f(a-b)=2d(|A|^2-|A|).
\end{equation}

On the other hand, using Lemma~\ref{complete}, we have
\begin{align*}
\sum_{a,b \in A} f(a-b)
&=\sum_{a, b\in A} \sum_{j=0}^{2d-1} c_j \overline{\chi^{j}(a-b)} \\
&=\sum_{a, b\in A} c_0 \overline{\chi^{0}(a-b)} + \sum_{j=1}^{2d-1} \sum_{a, b\in A}  c_j \overline{\chi^{j}(a-b)}  \\
&=d(|A|^2-|A|)+ \sum_{j=1}^{2d-1} \frac{c_j}{G(\chi^j)} \bigg(\sum_{c \in \F_q^*} \chi^j(c) |S(c)|^2 \bigg) \\
&=d(|A|^2-|A|)+ \sum_{j=1}^{2d-1} \frac{c_j}{G(\chi^j)} \bigg(\sum_{k=0}^{2d-1} \sum_{c \in C_k}\chi^j(c) |S(c)|^2 \bigg) \\
&=d(|A|^2-|A|) + \sum_{j=1}^{2d-1} \frac{c_j}{G(\chi^j)} \bigg(\sum_{k=0}^{2d-1} \theta^{jk}T_k \bigg). 
\end{align*}
By Corollary \ref{Gausssum}, $G(\chi^j)=-\sqrt{q}$ for each $1 \leq j \leq 2d-1$. Thus,
\begin{equation}\label{sumf}
\sum_{a,b \in A} f(a-b)
=d(|A|^2-|A|)-\frac{1}{\sqrt{q}}\sum_{k=0}^{2d-1} \bigg(\sum_{j=1}^{2d-1} c_j   \theta^{jk}\bigg) T_k.    
\end{equation}
Observe that
\begin{align}
\sum_{j=1}^{2d-1} c_j   \theta^{jk}
&=\sum_{j=1}^{2d-1} \sum_{\ell=1}^{d} \theta^{j m_\ell} \theta^{jk} =\sum_{\ell=1}^{d} \sum_{j=1}^{2d-1}  \theta^{j(m_\ell+k)} \nonumber\\
&=-d+\sum_{\ell=1}^{d} \sum_{j=0}^{2d-1}  \theta^{j(m_\ell+k)} \nonumber\\
&=
\begin{cases}
d, & \text{if } k \equiv -m_1,-m_2, \ldots, -m_d \pmod {2d},\\
-d,  & \text{otherwise} .
\end{cases}       \label{sumd}
\end{align}
For the last step, we used the following elementary fact based on the geometric series:
$$
\sum_{j=0}^{2d-1}\theta^{j \ell} = \begin{cases} 
0 &\text{if } \ell \not\equiv 0 \pmod{2d}, \\ 
2d &\text{if } \ell \equiv 0 \pmod{2d}.
\end{cases} 
$$
Combining equations~\eqref{sumf} and~\eqref{sumd}, we obtain that 
\begin{align*}
\sum_{a,b \in A} f(a-b)
&=d(|A|^2-|A|)+\frac{d}{\sqrt{q}}\sum_{k \neq -m_\ell} T_k -\frac{d}{\sqrt{q}} \sum_{\ell=1}^{d} T_{-m_\ell}\\
&\leq d(|A|^2-|A|)+\frac{d}{\sqrt{q}}\sum_{k=0}^{2d-1} T_k\\
&=d(|A|^2-|A|)+\frac{d}{\sqrt{q}} (q|A|-|A|^2). 
\end{align*}
In the last step, we used Lemma~\ref{sumc}. Comparing the inequality above with equation \eqref{clique}, 
it follows that for a clique $A$, we have $|A|^2-|A| \leq \frac{1}{\sqrt{q}} (q|A|-|A|^2)$, which implies that $|A| \leq \sqrt{q}$. This completes the proof of \eqref{item:trivial-upper-bound}. Moreover, the equality $|A|=\sqrt{q}$ holds if and only if 
\begin{equation}\label{Tzero}
T_{-m_1}=T_{-m_2}=\ldots=T_{-m_d}=0.
\end{equation}
This implies that $S(q,A;c)=0$ for all $c \in D'=\cup_{\ell=1}^{d} C_{-m_\ell}$, proving \eqref{item:characterization}.

\section{Algorithms and numerical computations}\label{sect:numerical-evidence}

This section presents algorithms to find and classify maximum cliques in semi-primitive pseudo-Paley graphs. We also report results from our experiments using SageMath~\cite{sagemath}  to support the conjectures in this paper. 

Using SageMath, we can find several examples of pseudo-Paley graphs of order $q$ coming from unions of semi-primitive cyclotomic classes with clique number $\sqrt{q}$. The following table summarizes some of these examples. The set $I$ indicates the choice of $C_i$'s in the connection set $D$. For example, the second row corresponds to the Cayley graph $\operatorname{Cay}(\F_{5^4}, D)$ where $D=C_0\cup C_1\cup C_3$ and 
$C_j = \{g^{6i+j} \ | \ 0\leq i\leq \frac{5^4-1}{6}\}$ and $g$ is a primitive root of $\F_{5^4}$. 

\begin{table}[ht!]
\begin{tabular}{ |c|c|c|c| } 
 \hline
 $q$ & $d$ & $I$ & $\omega(X)$\\
 \hline
 $5^4$ & $3$ & $\{0, 1, 3\}$ & $25$ \\ 
 \hline
 $7^4$ & $4$ & $\{0, 1, 2, 4\}$  & $49$ \\ \hline 
 $7^4$ & $4$ & $\{0, 1, 3, 6\}$  & $49$ \\
 \hline
 $3^8$ & $5$ & $\{0, 1, 2, 3, 5\}$  & $81$ \\
 \hline 
 $3^8$ & $5$ & $\{0, 1, 3, 6, 7\}$  & $81$ \\
 \hline 
   $11^4$ & $6$ & $\{0, 1, 2, 3, 6, 7\}$  & $121$ \\
   \hline
   $11^4$ & $6$ & $\{0, 1, 2, 4, 5, 8\}$  & $121$ \\
 \hline 
  $11^4$ & $6$ & $\{0, 1, 2, 3, 5, 10\}$  & $121$ \\
   \hline 
   $13^4$ & $7$ & $\{0, 1, 2, 3, 5, 6,9\}$  & $169$ \\
   \hline
   $13^4$ & $7$ & $\{0, 1, 2, 4, 7, 8, 10\}$  & $169$ \\
 \hline 
  $13^4$ & $7$ & $\{0, 1, 2, 5, 7, 9, 10\}$  & $169$ \\
   \hline 
\end{tabular}
\caption{\label{tab:evidence-main-conj} Graphs $X=PP(q,2d,I)$ with $\omega(X)=\sqrt{q}$}
\end{table}

In each of the graphs listed above, there are exactly $2$ cliques of size $\sqrt{q}$ containing $\F_{p^t}$, which supports Conjecture~\ref{conj:2-max-cliques}. Both cliques are subspaces over $\F_{p^t}$ which supports Conjecture~\ref{mainconj}. 

In general, finding the clique number of a Cayley graph is NP-hard \cite{GR17}. Computing the clique number of $X=PP(q,2d,I)$ for $q>2000$ with the current computational power is impossible. In contrast, finding the common neighbors of $\F_p$ inside $X$ takes only polynomial time, and SageMath can return the clique number of a graph with $<700$ vertices in a few minutes. Therefore, we employ the following algorithm.

\begin{algorithm}[H] \label{alg:clique-number}
$k=\min \{n \in \mathbb{N}: q/2^n<700\}$\\
$V=\{0,1,\ldots, k-1\} \cup \bigcap_{j=0}^{k-1} \text{Neighbors}(j)$ \ \ \ \ \ // Finding the set of common neighbors\\
$Y=PP(q,2d,I)[V]$ \ \ \ \ \ \ \ \ \ \ \ \  \ \ \ \ \ \ \ \ \ \ \ \ \ \ \ \ \ \ \ \ \ \ \ \  \ \ \ \ \ \ // Taking the induced subgraph \\
\Return $\omega(Y)$
\caption{Find the size of a maximum clique containing $\F_p$ inside $PP(q,2d,I)$ }
\end{algorithm}

Since the graph is $\frac{q-1}{2}$-regular, we expect that the subgraph $Y$ has roughly $q/2^k$ many vertices because the graphs we consider are quasi-random \cite{CGW89}.
Since $q/2^k<700$, SageMath can quickly return the clique number of $Y$. We used this algorithm to handle the exceptional cases in Theorem~\ref{thm: Peisert_clique_number}. 

In general, it suffices to consider maximum cliques containing $\{0, 1\}$. The reasoning is given at the beginning of the proof of Theorem~\ref{thm:clique_number}. When $q\leq 7^4$, we have $q/2^2 < 700$, and we can perform the algorithm outlined above with $k=2$ to find the clique number of all possible $PP(q, 2d, I)$. This verifies the main Conjecture~\ref{mainconj} for values of $q$ up to $7^4$.

In the final part of the paper, we present an efficient algorithm to verify Conjecture~\ref{conj:subspace} for a given prime $p$. The strategy is to fix an $\F_p$-subspace $V\subset \F_{p^4}$ with basis $\{1, a\}$ with $a\in C_0$
, and search for an index set $I\subset \{0, 1, \ldots, p\}$ with $|I|=(p+1)/2$ such that
$V \setminus \{0\} \subset \bigcup_{j\in I} C_{j}$. In order for $V$ to form a clique in $PP(p^4,p+1,I)$, by Theorem~\ref{theorem:main-result}, we must have
\begin{equation} \label{eq:I}
I=\big \{0 \leq j \leq p: |V \cap C_j|=\frac{p^2-1}{(p+1)/2}=2(p-1)\big\}.
\end{equation}
Since $1 \in V$, it follows that $0 \in I$. Note that $V \setminus \{0\}$ is a disjoint union of $(p+1)$ many $\F_p^{\ast}$-cosets with representatives $R=\{1, a, a+1, a+2, \ldots, a+(p-1)\}$. By Lemma~\ref{lem:C_0-contains-Fpt}, $\F_p^{\ast}$ is a subgroup of $C_0$, and so each cyclotomic class $C_j$ is a disjoint union of $\F_p^{\ast}$-cosets.
As a result, we can decompose $(V \setminus \{0\}) \cap C_j$ into $\F_p^{\ast}$-cosets. Therefore, equation~\eqref{eq:I} is equivalent to the statement that  $C_{j}$ contains exactly two $\F_{p}^{\ast}$-coset representatives in $R$ for each $j \in I$. Applying this idea, we can design a polynomial-time algorithm as follows: 

\begin{algorithm}[H] \label{alg:josh-question}
ValidSets $\gets \emptyset$ \\
\For{$k \gets 0$ to $p^2$}
{
$a \gets g^{(p+1)k}$\\
$L \gets []$ \ \ \ \ // initiating an empty list \\ 
$R \gets \{1,a,a+1, \ldots, a+(p-1)\}$ \\
   \For {$r \in R$}
   {
      $t\gets \log_g (r)$  \ \ \ \ \ // finding exponent $t$ with $g^t=r$ using the discrete logarithm \\
      $L\gets L \text{ append } \{t \pmod {p+1} \}$\\
   }
   $I\gets \operatorname{set}(L)$ \ \ \ \ \ \ \ \ \ \ \ \ // remove duplicates from the list $L$ to get the set $I$ \\
   $\operatorname{flag} \gets 1$ \ \ \ \  // assume $I$ is a valid set  \\
   \For {$j \in I$}
   {
      \If {$L.\operatorname{count}(j)\neq 2$}  
      {
          $\operatorname{flag} \gets 0$\ \ \ \ \ // \ $I$ is not valid   \\
          break   \\
      }
   }
   \If {$\operatorname{flag} =1$} 
   { 
      ValidSets $\gets$  ValidSets $\cup \{I\}$
   }
}
\Return $\#$ValidSets \ \ \ \ \ // return cardinality
\caption{Count the number of $I$ such that $0 \in I$ and $\omega(PP(p^4,p+1,I))=p^2$.}
\end{algorithm}

Using SageMath, Conjecture~\ref{conj:subspace} has been checked to hold for all odd primes $p<100$. We also found that each index set $I$ (except for $\{0,2, \ldots, p-1\}$) appeared exactly twice, which is consistent with Conjecture~\ref{conj:2-max-cliques}.

As an example, SageMath outputs the following index set when $p=97$ and $a=g^{p+1}=g^{98}$:
\begin{align*}
    I_0=\{0, 2, 3, 5, 6, 8, 9, 11, 14, 15, 16, 21, 22, 23, 24, 25, 26, 29, 32, 33, 35, 36, 43, 44, 45, \\ 46, 50, 51, 53, 56, 59, 62, 70, 72, 75, 76, 78, 80, 82, 84, 87, 88, 90, 92, 93, 94, 95, 96, 97\}
\end{align*}
In other words, $\omega(PP(97^4,98,I_0))=97^2$ with a maximum clique given by $V=\F_{97}\oplus g^{98}\F_{97}$.

\end{document}